\documentclass[12pt,reqno]{amsart}
\usepackage[T1]{fontenc}
\usepackage[english]{babel}
\usepackage{mathrsfs, csquotes}
\usepackage{amsfonts,amssymb,
amsmath,mathtools,amssymb, cancel}

\usepackage{color}
\definecolor{bleu_sombre}{rgb}{0,0,0.6}  \definecolor{rouge_sombre}{rgb}{0.8,0,0}\definecolor{vert_sombre}{rgb}{0,0.6,0}
\usepackage[plainpages=false,colorlinks,linkcolor=bleu_sombre,citecolor=rouge_sombre,urlcolor=vert_sombre,breaklinks]{hyperref}
\usepackage{enumerate}

\usepackage{caption}
\usepackage{subcaption}

\renewcommand{\leq}{\leqslant}	\renewcommand{\geq}{\geqslant}

\newcommand{\nb}{\nabla}

\newcommand{\Int}[2]{\displaystyle{\int_{#1}^{#2}}}

\newcommand{\Ab}{\mathbf {A}}

\newcommand{\Bb}{\mathbf {B}}
\newcommand{\ab}{\mathbf{a}}
\newcommand{\bb}{\mathbf{b}}
\newcommand{\R}{\mathbb{R}}

\newcommand{\nf}{\mathbf n}

\def\sig#1{\vbox{\hsize=5.5cm
		\kern2cm\hrule\kern1ex
		\hbox to \hsize{\strut\hfil #1 \hfil}}}

\renewcommand{\Re}{\mathrm{Re}\,}

\def\({\left(}
\def\){\right)}
\def\<{\left\langle}
\def\>{\right\rangle}

\def\n{\mathbf{n}}

\numberwithin{equation}{section}

\newcommand{\dd}{\mathrm{d}}

\newcommand{\be}{\begin{equation}}
\newcommand{\ee}{\end{equation}}
\newcommand{\bea}{\begin{eqnarray}}
\newcommand{\eea}{\end{eqnarray}}
\newcommand{\bee}{\begin{eqnarray*}}
\newcommand{\eee}{\end{eqnarray*}}
\newcommand{\curl}{{\rm curl\,}}

\def\n{\mathbf{n}}

\newtheorem{thm}{Theorem}[section]

\newtheorem{theorem}[thm]{Theorem}

\newtheorem{lemma}[thm]{Lemma}
\newtheorem{proposition}[thm]{Proposition}

\theoremstyle{remark}

\newtheorem{rem}[thm]{Remark}

\def\sig#1{\vbox{\hsize=5.5cm
\kern2cm\hrule\kern1ex
\hbox to \hsize{\strut\hfil #1 \hfil}}}
\numberwithin{equation}{section}

\begin{document}
\title[3D Magnetic Robin Laplacian]{Magnetic confinement for the 3D  Robin Laplacian}

\author[B. Helffer]{Bernard Helffer}
\address[Bernard Helffer]{Laboratoire de Math\'ematiques Jean Leray}
\email{bernard.helffer@univ-nantes.fr}
\author[A. Kachmar]{Ayman Kachmar}
\address[Ayman Kachmar]{Lebanese University, Department of Mathematics, Nabatieh, Lebanon.}
\email{ayman.kashmar@gmail.com}
\author[N. Raymond]{Nicolas Raymond}
\address[Nicolas Raymond]{Laboratoire Angevin de Recherche en Math\'ematiques}
\email{nicolas.raymond@univ-angers.fr}

\maketitle
\begin{abstract}
We determine   accurate asymptotics of the lowest eigenvalue  for the Laplace operator with a smooth magnetic field and Robin boundary conditions in a smooth 3D domain, when the Robin parameter tends to $+\infty$. Our results identify a critical regime where the contribution of the magnetic field and the Robin condition are of the same order. In this critical regime, we derive  an effective operator defined on the boundary of the domain.
\end{abstract}

\section{Introduction}

\subsection{Magnetic Robin Laplacian} 

We denote by $\Omega\subset\R^3$ a bounded   domain with a smooth boundary $\Gamma=\partial\Omega$. 
We study the lowest eigenvalue of the magnetic Robin
Laplacian in $L^2(\Omega)$,
\begin{equation}\label{Shr-op-Gen}
\mathcal{P}_\gamma=(-i\nabla+\ab)^2,
\end{equation}
 with domain
\begin{equation}\label{eq:bc}
D(\mathcal P_\gamma)=\{u\in H^2(\Omega)~:~i\, \nf\cdot(-i\nabla +\ab)u+\gamma\,u=0\quad{\rm on}~\partial\Omega\}\,.
\end{equation}
Here $\nf$ is the unit outward pointing normal vector of $\Gamma$, $\gamma>0$ the \emph{Robin} parameter and  $\ab\in C^2(\overline\Omega)$.  The vector field $\ab$ generates the  magnetic field
\begin{equation}\label{eq:A0}
\bb:=\curl\ab\in C^1(\overline\Omega)\,.
\end{equation}
Our hypotheses on $\ab$ and $\bb$ cover the physically interesting case of a uniform magnetic field of intensity $b$,  $\ab=\frac{b}2(-x_2,x_1,0)$ and $\bb=(0,0,b)$.

The operator $\mathcal P^b_\gamma$ is defined as the self-adjoint operator
associated with the following 
quadratic form (see, for instance, \cite[Ch.~4]{Hel}) 
\begin{equation}\label{QF-Gen}
H^1(\Omega)\ni u\mapsto  \mathcal{Q}^\ab_\gamma(u):=\int_\Omega\bigl|(-i\nabla+\ab) u(x)\bigr|^2\,dx-\gamma\Int{\Gamma}{}|u(x)|^{2}\dd s(x)\,.
\end{equation}
Our aim is to examine the magnetic effects on the principal eigenvalue 
\begin{equation}\label{eq:p-ev}
\lambda(\gamma,\bb)=\inf_{u\in H^1(\Omega)\setminus\{0\}}\frac{\mathcal{Q}^\ab_\gamma(u)}{\|u\|_{L^2(\Omega)}^2}
\end{equation} 
when  the Robin parameter $\gamma$ tends to $+\infty$. 

  The case $\gamma=0$ corresponds to the Neumann magnetic Laplacian, which has been studied in many papers \cite{HM3d, LP3d, R3d}.

\subsection{Mean curvature bounds}
In the case without  magnetic field, $\bb=0$, Pankrashkin and Popoff have proved in \cite{PP1} that, as $\gamma\to+\infty$, the lowest eigenvalue satisfies the following
\begin{equation}\label{as1}
\lambda(\gamma,0)=-\gamma^2-2\gamma \kappa _{\rm max}(\Omega)
+\mathcal O(\gamma^{2/3})\,,
\end{equation}
with 
\begin{equation}\label{defH}
\kappa _{\rm max}(\Omega):=\max\limits_{x\in\partial\Omega}\kappa _\Omega(x)\,,
\end{equation}
where $\kappa (x)=\kappa _\Omega(x)$ the mean curvature of $\partial\Omega$ at $x$.

The same asymptotic expansion continues to hold in the presence of a $\gamma$-independent magnetic field $\bb$. In fact, we have the non-asymptotic bounds
\begin{equation}\label{eq:PP-b=9}\lambda(\gamma,0)\leq \lambda(\gamma,\bb)\leq \lambda(\gamma,0)+ \|\ab\|^2_{L^\infty(\Omega)}\,.
\end{equation}
The lower bound is a simple consequence of the diamagnetic inequality, while the upper bound results by using the non-magnetic real eigenfunction (the eigenfunction corresponding to the eigenvalue $\lambda(\gamma,0)$)
as a test function for the quadratic form $\mathcal Q^\ab_\gamma$. Note incidently that the upper bound
 can be improved by minimizing over the $\ab$ such that $\curl \ab=\bb$.
 
  Consequently, we have
\begin{equation}\label{eq:lambda-b=0}
\lambda(\gamma,\bb)=-\gamma^2-2\gamma \kappa _{\rm max}(\Omega)
+\mathcal O(\gamma^{2/3})\,.
\end{equation}
It follows then, by an argument involving Agmon estimates, that the eigenfunctions concentrate near the set of points of maximal mean curvature, $\{\kappa_\Omega(x)=\kappa_{\rm max}(\Omega)\}$.

\subsection{Magnetic confinement}

The asymptotics expansion \eqref{eq:lambda-b=0} does not display  the contributions of the magnetic field, since the intensity of  the magnetic field is \emph{relatively} small. 

Magnetic effects are then expected  to appear in the large field limit, $\bb\gg 1$. We could start with the following rough lower bound, obtained by the diamagnetic inequality and the min-max principle,
\[ \lambda(\gamma,\bb)\geq (1-\delta)\lambda\Big(\frac{\gamma}{1-\delta},0\Big)+\delta\lambda(0,\bb)\qquad(0<\delta<1)\,,\]
which decouples the contributions coming from the large Robin parameter and the large magnetic field.  According to \eqref{as1}, the term $\lambda(\gamma,0)$ behaves like 
$-\gamma^2$ for large $\gamma$. The Neumann  eigenvalue $\lambda(0,\bb)$ was studied in \cite{HM3d}; it   behaves like   $\Theta_0b_0$ in the regime  
\[ b_0:=\inf\limits_{x\in\partial\Omega}\|\bb(x)\|\gg 1\,,\] 
where $\Theta_0\in(\frac 12,1)$ is  a universal constant  (the de Gennes constant). This comparison argument shows that the magnetic effects are   dominant when $ b_0 \gg \gamma^2$. In this case, the effective boundary condition is the Neumann condition ($\gamma=0$) and the role of the Robin condition  appear in the sub-leading terms (see \cite{K-dia, K-jmp} for the analysis of these effects in 2D domains).

Aiming to understand the competition  between the  Robin condition and the  magnetic field,  we  take the magnetic field parameter in the form
\begin{equation}\label{eq:b}
\bb=\gamma^\sigma \Bb\quad{\rm with}~0<\sigma<2\quad{\rm and}~\Bb\in C^1(\overline{\Omega})\,.
\end{equation}
Such competitions have been the object of investigations in the context of waveguides with Dirichlet boundary condition (see \cite{KR14}).

Our main results are summarized in the following theorems.
\begin{thm}\label{thm:ev}
Assume that \eqref{eq:b} holds. Then, as $\gamma\to+\infty$, the principal eigenvalue satisfies
\[\lambda(\gamma,\bb)=
-\gamma^2+\mathcal E(\gamma,\bb)+  o(\gamma^{\sigma})\,,
\]
where
\[\mathcal E(\gamma,\bb)=\min_{x\in\partial\Omega}\Big(|\bb\cdot\nf(x)|-  2 \kappa _\Omega(x)\gamma\Big)\,.\]
\end{thm}
\begin{rem}
This estimate in Theorem \ref{thm:ev} is also true for all the first eigenvalues.	
\end{rem}
\begin{rem}\label{rem:thm-ev}
The asymptotic result in Theorem~\ref{thm:ev} displays three regimes:
\begin{enumerate}[\rm (i)]
\item If $\sigma<1$, the magnetic field contribution is of lower order compared to that of the curvature, so the asymptotics in Theorem~\ref{thm:ev} reads 
\[ \lambda(\gamma,\bb)=-\gamma^2  - 2  \gamma\Big(\max_{x\in\partial\Omega} \kappa _\Omega(x)\Big)+o(\gamma)\,.\]
\item If $\sigma=1$, $\bb=\gamma \Bb$, the contributions of the magnetic field and the curvature are of the same order, namely
\[\lambda(\gamma,\bb)=
-\gamma^2+\gamma \min_{x\in\partial\Omega}\Big(|\Bb\cdot\nf(x)|-  2 \kappa _\Omega(x)\Big)+o\big(\gamma\big)\,.
\]
\item If $1<\sigma<2$, the contribution of the magnetic field is dominant compared to that of the curvature, so 
\[\lambda(\gamma,\bb)=
-\gamma^2+\gamma^\sigma  \min_{x\in\partial\Omega}|\Bb\cdot\nf(x)|+o\big(\gamma^\sigma\big)\,.
\]
\end{enumerate}
\end{rem}
Let us focus on the critical regime when $\sigma=1$. Under generic assumptions, an accurate (semiclassical) analysis of the first eigenvalues (establishing their simplicity) can be performed.

\begin{thm}\label{thm:ev-c}
Consider the regime $\sigma=1$ in \eqref{eq:b}. Assume that 
\[\partial\Omega\ni x\mapsto |\Bb\cdot\nf(x)|-  2\kappa _\Omega(x)\] 
has a unique and non-degenerate minimum, denoted by $x_0$  and that  
\begin{equation}\label{eq.gencond}
\Bb\cdot\nf(x_0)\not=0\,.
\end{equation}
Then, there exist $c_0>0$ and $c_1\in\mathbb{R}$ such that, for all $n\geq 1$,
\[\lambda_n(\gamma,\mathbf{b})=-\gamma^2+\gamma \left(|\Bb\cdot\nf(x_0)|-  2 \kappa _\Omega(x_0)\right)+(2n-1)c_0+c_1+\mathcal{O}(\gamma^{-\frac 12})\,.\]
Moreover, we have
\[c_0=\frac{\sqrt{\det(\mathrm{Hess}_{x_0}(|\mathbf{B}\cdot\mathbf{n}|-2\kappa_{\Omega} ))}}{2|\mathbf{B}\cdot\mathbf{n}(x_0)|} \,.\]
\end{thm}
\begin{rem}
Note that our assumption on the uniqueness of the minimum of the effective potential can be relaxed. Our strategy can deal with a finite number of non-degenerate minima. 
\end{rem}

Theorem~\ref{thm:ev-c} does not cover the situation of a uniform magnetic field and  constant curvature, since \eqref{eq.gencond} is not satisfied. Theorem~\ref{thm:ball-c-int} covers this situation, which displays a similar behavior to the one observed in \cite{HM3d, FP-ball}.    The contribution of the magnetic field  is related to  the ground state energy of the Montgomery  model \cite{Mont}
\[\nu_0:=\inf_{\zeta\in\R}\lambda(\zeta)\,,\]
 where
\[\lambda(\zeta)=\inf_{u\not=0}\int_\R\Big(|u'(s)|^2+\Big(\zeta+\frac{s^2}{2}\Big)^2|u(s)|^2 \Big)\,\dd s
\]

\begin{thm}\label{thm:ball-c-int}
Assume that $b>0$, $\Omega=\{x\in\R^3~:~|x|<1\}$ and the magnetic field is uniform and given by
\[
\bb=(0,0,\gamma b)\,.
\]
Then, as $\gamma\to+\infty$, the eigenvalue in \eqref{eq:p-ev} satisfies
\[\lambda(\gamma,\bb)=-\gamma^2-2\gamma+ \nu_0b^{4/3} \gamma^{2/3}+o(\gamma^{2/3})\,.\]
\end{thm}
\begin{rem}
We can expect that the expansion \enquote{of the form} given in Theorem \ref{thm:ball-c-int} is also true for a generic domain $\Omega$ when \eqref{eq.gencond} is not satisfied.	
\end{rem}

Comparing our results with their 2D counterparts \cite{K-dia, KS}, we observe in the 3D situation an effect due to  the \emph{magnetic geometry} which is not visible  in the 2D setting.  It can be explained as follows. The 2D case results from a cylindrical 3D domain with axis parallel to the magnetic field, in which case the term $\Bb\cdot\nf$ vanishes and the magnetic correction term will be of lower order compared to what we see in Theorem~\ref{thm:ev}.

\subsection{Structure of the paper}
 The paper is organized as follows. In Section~\ref{sec.2}, we introduce an effective semiclassical parameter, introduce auxiliary operators and eventually  prove Theorem~\ref{thm:ev}. In Section~\ref{sec:efop}, we derive an effective operator and then in Section~\ref{sec:efop*} we estimate the low-lying eigenvalues  for the effective operator, thereby proving Theorem~\ref{thm:ev-c}. Finally, in Appendix \ref{sec:ball}, we analyze  the case of the ball domain in the uniform magnetic field case and prove Theorem~\ref{thm:ball-c-int}. We also discuss in this appendix $\gamma$-independent  uniform fields (which amounts to considering the case $\sigma=0$ in \eqref{eq:b}).

\section{Proof of Theorem~\ref{thm:ev}}\label{sec.2}

\subsection{Effective operators}

\subsubsection{Effective 1D Robin Laplacian}\label{sec:ef-1D}~\\
 We fix three constants\footnote{ The constant $C_*$ depends on the local geometry of $\partial\Omega$ near some point $x_*\in\partial\Omega$, see \eqref{eq:jac}. By compactness of the boundary, $C_*$ can be selected independently of the choice of the boundary point $x_*$.}  $C_*>0$, 
 $\sigma\in(0,2)$, and $h>0$ (the so-called semiclassical parameter). We set
\begin{equation}\label{eq:delta}
\delta=h^{\rho-\frac1\sigma} \quad{\rm with }~0<\rho<\frac12\,.
\end{equation}
For every $x_*\in\partial\Omega$, we introduce the effective \emph{transverse} operator 
\begin{equation}\label{eq:1D-op}
L_*:=-h^2\big(w_*(t)\big)^{-1}\frac{\dd}{\dd t}\left(w_*(t)\frac{\dd}{\dd t}\right)
\end{equation}
in the weighted Hilbert space $L^2\Big((0,h^\rho);w_*\dd t\Big)$,   where
\[w_*(t)=1-2\kappa (x_*)t-C_*t^2\,,\]
and the domain of $L_*$ is
\[D(L_*)=\{u\in H^2(0,h^\rho)~:~u'(0)=-h^{-\frac1\sigma}u(0)~\&~u(h^\rho)=0\}\,.\]
The change of variable, $\tau=h^{-\frac1\sigma}t$, yields the new operator
\begin{equation}\label{eq:1D-opa}
\widetilde L_*:=-h^{2-\frac{2}\sigma}\big(w_{*,h}(\tau)\big)^{-1}\frac{\dd}{\dd\tau}\left(w_{*,h}(\tau)\frac{\dd}{\dd\tau}\right)
\end{equation}
with domain
\[ D(\widetilde L_*)=\{u\in H^2(0,\delta)~:~u'(0)=-u(0)~\&~u(\delta)=0\}\,.\]
The new weight $w_{*,h}$ is defined as follows
\[w_{*,h}(\tau)=1-2\kappa (x_*)h^{\frac{1}\sigma} \tau-C_*h^{\frac2\sigma}\tau^2\,.\]
Using \cite[Sec.~4.3]{HK-tams}, we get that the first eigenvalue of the operator $L_*$ satisfies, as $h\to0_+$,
\begin{equation}\label{eq:1D-eff-n}
\lambda(L_*)=h^{2-\frac{2}\sigma}\lambda(\widetilde L_*)=
-h^{2-\frac{2}\sigma}-2\kappa (x_*)h^{2-\frac1\sigma}+\mathcal O(h^{2})\,.
\end{equation}

\subsubsection{Effective harmonic oscillator}~\\
We also need the family of harmonic oscillators in $L^2(\R)$,
\begin{equation}\label{eq:h-osc}
T_{m,\xi}^{h,\eta}:=(-ih\partial_s-m)^2+(\xi+m+\eta s)^2\,,
\end{equation}
where $(m,\xi,\eta)$ are parameters.

By a gauge transformation  and a translation (when $\eta\not=0$), we observe that the first eigenvalue of $T_{m,\xi}^{h,\eta}$ 
is independent of $(m,\xi)$. By rescaling, and using the usual harmonic oscillator, we see that the first eigenvalue is given by\footnote{We will use the inequality $\lambda \big(T_{m,\xi}^{h,\eta}\big)\geq |\eta| h$ (which is obvious when $\eta=0$).}
\begin{equation}\label{eq:h-osc*}
\lambda \big(T_{m,\xi}^{h,\eta}\big)=|\eta| h \,.
\end{equation}

\subsubsection{Effective semiclassical parameter}~\\
In our context,  the semiclassical parameter will be 
\begin{equation}\label{eq:h}
h=\gamma^{-\sigma}\quad{\rm with }~0<\sigma<2\,.
\end{equation}
Under the assumption in \eqref{eq:b}, the quadratic form in \eqref{QF-Gen} is expressed as follows
\[
Q_\gamma^\ab(u)=h^{-2} q_h(u)\,,
\]
where
\begin{equation}\label{eq:qh}
q_h(u)=\int_\Omega|(-ih\nabla+\Ab)u|^2\dd x-h^\alpha\int_{\partial\Omega}|u|^2\dd x\,,
\end{equation}
$\curl\Ab=\Bb$ is a fixed vector field, and 
\begin{equation}\label{eq:alpha}
\alpha=\alpha(\sigma):=2-\frac1\sigma\in  \left(-\infty,\frac 32 \right)\,.
\end{equation}
We introduce the eigenvalue
\begin{equation}\label{eq:mu-h}
\mu(h,\Bb)=\inf_{u\in H^1(\Omega)\setminus\{0\}}\frac{q_h(u)}{\|u\|_{L^2(\Omega)}^2}\,.
\end{equation}
Then we have the relation
\begin{equation}\label{eq:ev-g-h}
h=\gamma^{-\sigma}\quad{\rm and}\quad \lambda(\gamma,\bb)=h^{-2}\mu(h,\Bb) \,.
\end{equation}

~\\
\subsection{Local boundary coordinates}
 We follow the presentation in \cite{HM3d}.
\subsubsection{The coordinates}~\\
We fix $\epsilon>0$ such that the distance function
\begin{equation}\label{eq:t(x)}
t(x)={\rm dist}(x,\partial\Omega)
\end{equation}
is smooth in $\Omega_\epsilon:=\{{\rm dist}(x,\Omega)<\epsilon\}$.

Let $x_0\in\partial\Omega$ and choose a chart  $\Phi:V_0\to \Phi(V_0)\subset\partial\Omega$ such that $x_0\in\Phi(V_0)$ and $V_0$ is an open subset of $\R^2$. We set  
\begin{equation}\label{eq:y0,W}
y_0= \Phi^{-1}(x_0)\quad{\rm and } \quad W_0=\Phi(V_0)\,.
\end{equation} 
We denote by 
\[G=\sum_{1\leq i,j\leq 2}G_{ij}\, dy_i\otimes dy_j\]
 the metric  on the surface $W_0$ induced by the Euclidean metric, namely $G=(\dd\Phi)^{\mathrm{T}} \dd\Phi$. After a dilation and a translation of the $y$ coordinates, we may assume that
\begin{equation}\label{eq:y0,W*}
y_0=0\quad{\rm and}\quad G_{ij}(y_0)=\delta_{ij}\,.
\end{equation}
We introduce the new coordinates $(y_1,y_2,y_3)$   as follows
\[\tilde\Phi: (y_1,y_2,y_3)\in V_0\times(0,\epsilon)\mapsto \Phi(y_1,y_2)-t\, \nf\big(\Phi(y_1,y_2) \big)\,,\]
and we set
\begin{equation}
U_0=\tilde\Phi\big( V_0\times(0,\varepsilon)\big)\subset\Omega\,.
\end{equation}
Note that $y_3$ denotes the normal variable in the sense that for a point $x\in U_0$ such that $(y_1,y_2,y_3)=\tilde\Phi^{-1}(x)$, we have  $y_3=t(x)$ as introduced in \eqref{eq:t(x)}. In particular, $y_3=0$ is the equation of the surface $U_0\cap\partial\Omega\,$.

\subsubsection{Mean curvature}
We denote by $K$ and $L$  the second and third fundamental forms on $\partial\Omega$. In the coordinates $(y_1,y_2)$ and with respect to the canonical basis, their matrices are given by
\begin{equation}\label{eq:F-forms}
K=\sum_{1\leq i,j\leq 2}K_{ij}\,dy_i\otimes dy_j
\quad{\rm and}\quad L=\sum_{1\leq i,j\leq 2} L_{ij}\, dy_i\otimes dy_j\,,
\end{equation}
where
\[K_{ij}=\left\langle\frac{\partial x}{\partial y_i},\frac{\partial\nf}{\partial y_j} \right\rangle
\quad{\rm and}\quad  L_{ij}=\left\langle\frac{\partial \nf}{\partial y_i},\frac{\partial\nf}{\partial y_j} \right\rangle\,.\]
The mean curvature $\kappa $ is then defined as half the trace of the matrix of $G^{-1}K=(k_{ij})_{1\leq i,j\leq 2}$. For $x=\Phi(y_1,y_2)$, we have
\begin{equation}\label{eq:mc-H}
\kappa (x)=\frac12{\rm tr}(G^{-1}K)\Big|_{(y_1,y_2)}=\frac12\big(K_{11}(y_1,y_2)+K_{22}(y_1,y_2)\big)\,.
\end{equation} 
In light of \eqref{eq:y0,W} and \eqref{eq:y0,W*}, we write 
\begin{equation}\label{eq:mc-Ha}
\kappa (x_0)=\frac12\big(K_{11}(0)+K_{22}(0)\big)\,.
\end{equation}

\subsubsection{The metric}~\\
The Euclidean metric  $g_0$ in $\R^3$ is block-diagonal in the new coordinates and takes the form (see \cite[Eq.~(8.26)]{HM3d}) 
\begin{equation}\label{eqg0}
\begin{aligned}
 g_0=(\dd\tilde\Phi)^{\mathrm{T}}\dd\tilde\Phi&=\sum_{1\leq i,j\leq 3}g_{ij}\, dy_i\otimes dy_j\\
&= dy_3\otimes dy_3+\sum_{1\leq i,j\leq 2}\big(G_{ij}(y_1,y_2)-2y_3K_{ij}(y_1,y_2)+y_3^2L_{ij}(y_1,y_2)\big)dy_i\otimes dy_j\,,
\end{aligned}
\end{equation}
where $(K_{ij})$ and $(L_{ij})$ are defined in \eqref{eq:F-forms}.
Our particular choice of the coordinates, together with $G(0)={\rm Id}$ (see \eqref{eq:y0,W*}), yields
\begin{equation}\label{eq:g,ij}
g_{ij}=\begin{cases}
0 &~{\rm if~} (i,j) \in\{(3,1),(3,2),(1,3),(2,3)\}\\
 \delta_{ij} +\mathcal O(|y|)&~{\rm if~} 1\leq i,j\leq 2\\
1&{~\rm if~}i=j=3
\end{cases}
\,.
\end{equation}
The coefficients of $(g^{ij})$, the inverse matrix of $(g_{ij})$, are  then given as follows
\begin{equation}\label{eq:g,ij*}
g^{ij}=\begin{cases}
0 &~{\rm if~} (i,j) \in\{(3,1),(3,2),(1,3),(2,3)\}\\
 \delta_{ij}+\mathcal O(|y|)&~{\rm if~} 1\leq i,j\leq 2\\
1&{~\rm if~}i=j=3
\end{cases}\,.
\end{equation}
 We denote by $g=(g_{ij})$ the matrix of the metric $g_0$ in the $y$ coordinates; the determinant of $g$ is  denoted by $|g|$; we then have
\begin{align*}
|g|^{1/2}&=\Big({\rm det}(G-y_3\,K+y_3^2\,L)\Big)^{1/2}\\
&=\Big({\rm det}(I-y_3\,G^{-1}K+y_3^2\,G^{-1}L)\Big)^{1/2} |G|^{1/2}\\
 &=\big(1-y_3 {\rm tr}(G^{-1}K) +y_3^2 { p_2(y)}\big)|G|^{1/2}\,,
\end{align*}
where  $p_2$ is a bounded function  in the neighborhood $V_0\times[0,\varepsilon]$.\\
In light of \eqref{eq:mc-H}, we infer the following important inequalities which involve the mean curvature $\kappa $,
valid in  $V_0\times[0,\varepsilon]$,
\begin{equation}\label{eq:jac}
\big(1-2y_3 \kappa \big(\Phi(y_1,y_2)\big)-C_* y_3^2\big)|G|^{1/2}\leq 
|g|^{1/2}\leq \big(1-2y_3 \kappa \big(\Phi(y_1,y_2)\big)+C_* y_3^2\big)|G|^{1/2}\,,
\end{equation}
with $C_*$ a constant independent of $y$.

\subsubsection{The magnetic potential}~\\
The reader is referred to \cite[Section 0.1.2.2]{R3d}. We recall that
\[\sigma_\Ab=\sum_{i=1}^3A_i\dd x_i\,,\]
so that, the change of coordinates $x=\tilde\Phi(y)$ gives
\[\tilde\Phi^*\sigma_{\Ab}=\sum_{i=1}^3\tilde A_i\dd y_i\,,\quad \tilde\Ab=(\dd\tilde\Phi)^{\mathrm{T}}\circ\Ab\circ\tilde\Phi(y)\,.\]
The magnetic field $\Bb$ is then defined via the $2$-form
\[\omega_\Bb:=\dd\sigma_\Ab=\sum_{1\leq i,j\leq 3}  b_{ij}\dd x_i\wedge \dd x_j\quad{\rm where}~b_{ij}=\frac{\partial A_j}{\partial x_i}-\frac{\partial A_i}{\partial x_j}\,.\]
The 2-form $\omega_\Bb$ can be viewed as the vector field $\Bb$ given by
\[ \Bb=\sum_{i=1}^3 B_i\frac{\partial}{\partial x_i}~{\rm with~}B_1=b_{23},~B_2=b_{31},~B_3=b_{12}\,,\]
via the Hodge identification
\[\omega_{\Bb}(u,v)=\langle u\times v,\Bb\rangle_{\R^3}\,.\]
We have that \[\tilde\Phi^*\omega_{\Bb}=\dd(\tilde\Phi^*\sigma_{\Ab})=\sum_{1\leq i,j\leq 3}  \left(\frac{\partial \tilde A_j}{\partial y_i}-\frac{\partial \tilde A_i}{\partial y_j}\right)\dd y_i\wedge \dd y_j\,.\]
Considering the magnetic field $\mathcal{B}$ associated with $\tilde\Ab$, this means that
\[\omega_{\Bb}(\dd\tilde\Phi(u),\dd\tilde\Phi(v))=\omega_{\mathcal{B}}(u,v)\,,\quad \mbox{or }\quad\langle\dd\tilde\Phi(u)\times\dd\tilde\Phi(v),\Bb\rangle=\langle u\times v,\mathcal{B}\rangle\,,\]
or, equivalently,
\[\det(\dd\tilde\Phi)\langle u\times v,\dd\tilde\Phi^{-1}(\Bb)\rangle=\langle u\times v,\mathcal{B}\rangle\,,\]
\emph{i.e.},
\begin{equation}\label{eq.tildeBcalB}
\tilde\Bb:=\dd\tilde\Phi^{-1}(\Bb)=\det(\dd\tilde\Phi)^{-1}\mathcal{B}\,.
\end{equation}
Explicitly,
\begin{equation}\label{eq:B-y}
\tilde B_1=|g|^{-1/2}\left(\frac{\partial \tilde A_3}{\partial y_2} -\frac{\partial\tilde A_2}{\partial y_3} \right)\,,~ 
\tilde B_2=|g|^{-1/2}\left(\frac{\partial \tilde A_1}{\partial y_3} -\frac{\partial\tilde A_3}{\partial y_1} \right)\,,~
\tilde B_3=|g|^{-1/2}\left(\frac{\partial \tilde A_2}{\partial y_1} -\frac{\partial\tilde A_1}{\partial y_2} \right)\,,
\end{equation}
and $\tilde\Bb$ is the vector of coordinates of $\Bb$ in the new basis induced by $\tilde\Phi$. We remark for further use that
\begin{equation}\label{eq:Bn}
\Bb\cdot \nf=-\tilde B_3\,.
\end{equation}
We can use a (local) gauge transformation, $\Ab\mapsto \Ab^{\phi}:=\Ab+\nabla \phi$, and obtain that the normal component of $\Ab^\phi$, $\tilde A^\phi_3$, vanishes. We assume henceforth 
\begin{equation}\label{eq:A3=0}
\tilde A_3=0\,.
\end{equation}

\subsubsection{The quadratic form}~\\
For  $u\in H^1(\Omega)$, we introduce the \emph{local} quadratic form 
\begin{equation}\label{eq:loc-qf}
q_h(u;U_0)=\int_{U_0}|(-ih\nabla+\Ab)u|^2\dd x-h^\alpha\int_{U_0\cap\partial\Omega}|u|^2\dd x\,.
\end{equation}
In the new coordinates $y=(y_1,y_2,y_3)$, we express the quadratic form as follows
\begin{multline}\label{defqu}
q_h(u;U_0)=\int_{V_0\times(0,\varepsilon)}|g|^{1/2}\sum_{1\leq i,j\leq 3}g^{ij}(hD_{y_i}+\tilde A_i)\tilde u \overline{(hD_{y_j}+\tilde A_j)\tilde u}\dd y\\
-h^\alpha\int_{V_0}|\tilde u(y_1,y_2,0)|^2|G|^{1/2}\dd y_1\dd y_2
\end{multline}
where $D_{y_i}=-i\frac{\partial}{\partial y_i}$, the coefficients $g^{ij}$ are introduced in \eqref{eq:g,ij*} and  $\tilde u=u\circ \tilde\Phi$. 

\begin{rem}\label{rem:qf-b-c}
The formula in \eqref{defqu} results from the following identity
\begin{equation}\label{eq:grad-y-c}
|(-ih\nabla +\Ab)u|^2=\sum_{i,j=1}^3 g^{ij} (hD_{y_i}+\tilde A_i)\tilde u \overline{(hD_{y_j}+\tilde A_j)\tilde u}\,.
\end{equation}
Now \eqref{defqu} follows. Using \eqref{eq:grad-y-c} for $\Ab=0$  and using \eqref{eq:g,ij*}, we observe that
\[|\nabla u|^2\leq m |\nabla_y \tilde u|^2\,,\]
for a positive constant $m$, which we can choose independently of the point $x_0$, by compactness of $\partial\Omega$. Also, if we denote by $\nabla'$  the gradient on $\partial\Omega$, and if $u$ is independent of the distance to the boundary (i.e. $\partial_{y_3}\tilde u=0$), we get
\begin{align*} 
|\nabla u|^2&= |\nabla'u|^2+ \sum_{i,j=1}^2 \Big(g^{ij}-{g^{ij}}_{/_{y_3=0}}\Big) \partial_{y_i}\,\tilde u\overline{\partial_{y_j}\tilde u}\\
&\leq \big(1+M y_3\big) |\nabla'u|^2\,,
\end{align*} 
where  we used \eqref{eq:g,ij*}, and $M$ is positive constant. 
\end{rem}

We assume that
\[\rho\in\left(0,\frac12\right)\,.\]
This condition appears later in an argument involving a partition of unity, where we encounter an error term of the order $h^{2-2\rho}$ which we require to be $o(h)$ (see \eqref{eq:qh-ims}).

Now we fix some constant $c_0$ so that
\[ \Phi^{-1}\big( B(x_0,2h^\rho)\cap\partial\Omega\big)\subset \{|y|<  c_0 \, h^\rho\} \,.\]
 We infer from \eqref{eq:g,ij*} and \eqref{eq:jac} that  when ${\rm supp}\,\tilde u\subset\{|y|< c_0 h^\rho\}$, 
\begin{equation}\label{eq:loc-qf-lb}
q_h(u)\geq  q_h^{\rm tran}(\tilde u) 
+
(1-Ch^\rho)\, q_h^{\rm surf}(\tilde u)
\end{equation}
where 
\begin{equation}\label{eq:loc-qf-lb1}
 q_h^{\rm tran}(\tilde u) = \int_{V_0}\left(\int_{(0,\varepsilon)} w_*(y)| h \partial_{y_3} \tilde u|^2 \dd y_3-h^\alpha|\tilde u(y_1,y_2,0)|^2\right)|G|^{1/2}\dd y_2\dd y_3\,,
 \end{equation}
\begin{equation}\label{eq:loc-qf-lb2}
q_h^{\rm surf}(\tilde u)=\sum_{i=1}^2
\int_{V_0\times(0,\varepsilon)}|(hD_{y_i}+  \tilde A_i) \tilde u|^2 \dd y\,,
\end{equation}
and 
\begin{equation}\label{eq:loc-qf-lb3}
w_*(y)=1-2y_3\tilde \kappa (y_1,y_2) -C_*y_3^2\quad{\rm with }\quad \tilde \kappa =\kappa \circ \Phi \,.
\end{equation}
Note that we used the Cauchy-Schwarz inequality to write that, if $m_{ij}=\delta_{ij}+\mathcal O(h^\rho)$, then, for some constant $C_2>0$, 
\[(1-C_2h^\rho)\sum_{i=1}^2|d_i|^2\leq  \sum_{1\leq i,j\leq 2}m_{ij}d_id_j\leq (1+C_2h^\rho)\sum_{i=1}^2|d_i|^2\,.\]
\begin{equation}\label{eq:loc-qf-lb*}
q_h^{\rm tran}(\tilde u)\geq \sum_{i=1}^2 \int_{V_0\times (0,\varepsilon)} \Big(-h^{2-\frac{2}\sigma}-2\tilde \kappa (y_1,y_2)\, h^{2-\frac1\sigma}+\mathcal O(h^{2})\Big)|\tilde u|^2 |g|^{1/2}\dd y\,.
\end{equation}
In the sequel, we will estimate the term \eqref{eq:loc-qf-lb2}
\[q^{\rm surf}_h(\tilde u)= \sum_{i=1}^2\int_{\{|y|<c_0h^\rho, y_3>0\}}|(hD_{y_i}+\tilde A_i)\tilde u|^2 dy\,.\]
We write the Taylor  expansion at $0$ of $\tilde A_i$ (for $i=1,2$) to order $1$,
\begin{equation}\label{eq.Taylor2}
\tilde A_i(y)=\tilde A^{\rm lin}_i(y)+\mathcal O(|y|^2)
\end{equation}
where
\[\tilde A_i^{\rm lin}(y)=\tilde A_i(0)+y_1\frac{\partial\tilde A_i}{\partial y_1}(0)+y_2\frac{\partial\tilde A_i}{\partial y_2}(0)+
y_3\frac{\partial\tilde A_i}{\partial y_3}(0)\,.\]
We set $\tilde A^{\rm lin}(y)=\big(\tilde A_1^{\rm lin}(y),\tilde A_2^{\rm lin}(y)\big)$ and observe by \eqref{eq:B-y} that
\begin{equation}\label{eq:A-lin-v1}
\tilde A^{\rm lin}(y)=\Big(-\mathcal{B}_2^0y_3\,,\mathcal{B}_3^0y_1, -\mathcal{B}_1^0y_3 \Big)+\nb_{(y_1,y_2)}w\,,
\end{equation}
\[\mathcal{B}_i^0=\mathcal{B}_i(0)\,,\quad \mathcal{B}=\curl\tilde\Ab\,,\]
where
 \begin{equation}\label{defw}
 \, w(y_1,y_2)=\tilde A_1(0)y_1+\tilde A_2(0) y_2+a_{11}\frac{y_1^2}{2}+a_{12}y_1y_2+a_{22}\frac{y_2^2}{2}
 \end{equation}
  and 
\begin{equation}a_{ij}=\frac{\partial \tilde A_i}{\partial y_j}(0)\,.
\end{equation}
So, after a gauge  transformation, we may assume that
\begin{equation}\label{eq:A-lin-fv}
\tilde A^{\rm lin}(y)=\Big(-\mathcal{B}_2^0y_3\,,\, \mathcal{B}_3^0y_1\,,\,-\mathcal{B}_1^0y_3 \Big)\,.
\end{equation}
Now we estimate from below the quadratic form by Cauchy's inequality and obtain, for all $\zeta\in(0,1)$,
\[q^{\rm surf}_h(\tilde u)\geq  (1-\zeta)\sum_{i=1}^2\int_{{\{|y|<c_0h^\rho,y_3>0\}}}|(hD_{y_i}+\tilde A_i^{\rm lin})\tilde u|^2 \dd y-C\zeta^{-1}h^{4\rho}\int_{{\{|y|<c_0h^\rho,y_3>0\}}}|\tilde u|^2\dd y\,.\]
We do a partial Fourier transformation with respect to the variable $y_2$ and eventually we get
\begin{multline*}
q^{\rm surf}_h(\tilde u)\geq  (1-\zeta)\int_\R\Big(\int_{\{|(y_1,y_2)|<c_0h^\rho\}}\Big(|(hD_{y_1}-\mathcal{B}_2^0y_3)\hat u|^2\\
+|(\xi -\mathcal{B}_1^0 y_3+\mathcal{B}_3^0y_1)\hat u|^2\Big)\dd y_2\dd y_1\Big)\dd\xi
-C\zeta^{-1}h^{4\rho}\int_{{\{|y|<c_0h^\rho,y_3>0\}}}|\tilde u|^2\dd y\,.
\end{multline*}
Using \eqref{eq:h-osc*}, we get 
\begin{equation}\label{eq:q-surf*}
\begin{aligned}
 q^{\rm surf}_h(\tilde u)&\geq (1-\zeta) \int_{\{V_0\times(0,\varepsilon)\}}\Big( |\mathcal{B}_3^0|h-C\zeta^{-1}h^{4\rho}\Big)|\tilde u|^2\dd y\\
&\geq \big(1-C\zeta-Ch^\rho\big) \int_{\{V_0\times(0,\varepsilon)\}}\Big( |\mathcal{B}_3^0||g(0)|^{-\frac 12}h-C\zeta^{-1}h^{4\rho}\Big)|\tilde u|^2 |g|^{1/2} \dd y \,.
 \end{aligned}\end{equation}
We now choose 
\[\zeta=h^{\rho}\quad \mbox{ and }\quad \rho=\frac2{5}  \,.\]
Collecting \eqref{eq:loc-qf-lb*}, \eqref{eq:q-surf*}, \eqref{eq:loc-qf-lb}, \eqref{eq.tildeBcalB}, and \eqref{eq:Bn}, and then returning to the Cartesian coordinates, we get
\begin{equation}\label{eq:loc-qf**}
q_h(u)\geq \int_{\Omega}\Big(-h^{2-\frac{2}\sigma}-2\kappa (x_0)h^{2-\frac1\sigma}+ |\Bb\cdot\nf(x_0)|h-Ch^{6/5}\Big) | u|^2\dd x \,.
\end{equation}
 for $u\in H^1(\Omega)$ with support in a ball $B(x_0,h^{2/5})\cap \overline{\Omega}$. Moreover, using the compactness of $\partial\Omega$, we can choose the constant $C$ in \eqref{eq:loc-qf**} independent of $x_0\in\partial\Omega$.

\begin{rem}\label{rem:qf-ub}
We can write an upper bound of the quadratic form  similar to the lower bound in \eqref{eq:loc-qf-lb}. In fact,  assuming that $u\in H^1(\Omega)$ with ${\rm supp}\,\tilde u\subset \{|y|<c_0h^\rho\}$, then using \eqref{eq:g,ij*} and \eqref{eq:jac}, we get, with the notation in \eqref{eq:loc-qf-lb2},
\begin{equation}\label{eq:loc-qf-ub}
q_h(u)\leq \bar q_h^{\rm tran}(\tilde u)+
(1+Ch^\rho) q_h^{\rm surf}(\tilde u)\,,
\end{equation}
where
\begin{equation}\label{eq.qtildetran}
\bar q_h^{\rm tran}(\tilde u) = \int_{V_0}\left(\int_{(0,\varepsilon)}  w^*(y)|h\partial_{y_3} \tilde u|^2 \dd y_3-h^\alpha|\tilde u(y_1,y_2,0)|^2\right)|G|^{1/2}\dd y_2\dd y_3\,,
\end{equation}
 and
 \[w^*(y)=1-2y_3\tilde \kappa (y_1,y_2) +C_*y_3^2\,.\]

\end{rem}
\subsection{Lower bound}~\\
Using \eqref{eq:loc-qf**}, we get by a standard covering argument involving a partition of a unity (see \cite[Sec.~7.3]{HM3d}), the following lower bound on the eigenvalue $\mu(h,\Bb)$,  
\begin{equation}\label{eq:mu-lb}
\mu(h,\Bb)\geq -h^{2-\frac{2}\sigma}+\min_{x_0\in\partial\Omega}\big( |\Bb\cdot\nf(x_0)|h-2\kappa (x_0)h^{2-\frac1\sigma}\big)- C h^{6/5}\,.
\end{equation}
This yields the lower bound in Theorem~\ref{thm:ev}, in light of the relation between the eigenvalues $\mu(h,\Bb)$ and $\lambda(\gamma,\bb)$ displayed in \eqref{eq:ev-g-h}.

Let us briefly recall how to get \eqref{eq:mu-lb}. Let $\rho=\frac2{5}$. Consider a partition of unity of $\overline\Omega$
\[\varphi_{1,h}^2(x)+\varphi_{2,h}^2(x)=1\] 
with the property that, for some $h_0>0$, there exists $C_0$ such that, for $h\in (0,h_0]$, 
\[{\rm supp\,}\varphi_{1,h}\subset \{{\rm dist}(x,\partial\Omega)>\frac12h^\rho)\},\quad  
{\rm supp\,}\varphi_{2,h} \subset \{{\rm dist}(x,\partial\Omega)<h^\rho)\} \quad{\rm and}\quad |\nabla\varphi_{i,h}|\leq C_0h^{-\rho}\,.\]
We  decompose the quadratic form in \eqref{eq:qh}, and get, for  $u\in H^1(\Omega)$,
\begin{equation}\label{eq:qh-ims}
q_h(u)=\sum_{i=1}^2 \Big(q_h(\varphi_{i,h}u)-h^2\|\,|\nabla\varphi_{i,h}|u\,\|^2 \Big)\geq  q_h(\varphi_{2,h}u)-C_0^2h^{2-2\rho}\|u\|^2\,.
\end{equation}
Now we introduce a new partition of unity such that 
\[\sum_{j=1}^N\chi_{j,h}^2=1 \mbox{ on } \{x\in\overline{\Omega},~{\rm dist}(x,\partial\Omega)<h^\rho \}\]
where
\[{\rm supp\,}\chi_{j,h}\subset B(x_{j0},2 h^\rho)\cap\overline{\Omega}\quad(x_{j0}\in\partial\Omega)\,,\]
and 
\[|\nabla \chi_{j,h}|\leq \bar C_0 \, h^{-\rho}\,.\]
Again, we have the decomposition formula
\[ q_h(\varphi_{1,h}u)=\sum_{j=1}^N \Big(q_h(\varphi_{2,h}\chi_{j,h}u)-h^2\|\,|\nabla \chi_{j,h}|u\|^2\Big)\geq \sum_{j=1}^N q_h(\varphi_{2,h}\chi_{j,h}u)-\bar C_0h^{2-2\rho}\|u\|^2\,.\]
We estimate $q_h(\varphi_{2,h}\chi_{j,h}u)$ from below using \eqref{eq:loc-qf**} and  we get
\begin{multline*}
q_h(\varphi_{1,h}u)\geq \sum_{j=1}^N \int_{\Omega}\Big(-h^{2-\frac{2}\sigma}-2\kappa (x_{j0})h^{2-\frac1\sigma}+ |\Bb\cdot\nf(x_{j0})|h-Ch^{6/5}\Big) | \varphi_{2,h}\chi_{j,h} u|^2\dd x\\
-\bar C_0h^{2-2\rho}\|u\|^2\,.
\end{multline*}
Now we use that, for $h$ sufficiently small
\begin{multline*}
-h^{2-\frac{2}\sigma}-2\kappa (x_{j0})h^{2-\frac1\sigma}+ |\Bb\cdot\nf(x_{j0})|h-Ch^{6/5}\\
\geq -h^{2-\frac{2}\sigma}+\min_{x_0\in\partial\Omega}\big( |\Bb\cdot\nf(x_0)|h-2\kappa (x_0)h^{2-\frac1\sigma}\big)-Ch^{6/5}\,.
\end{multline*}

In this way we infer from \eqref{eq:qh-ims}   and the fact that $\rho=\frac2{5}$,
\begin{multline*}
q_h(u)\geq \int_{\Omega}\Big(-h^{2-\frac{2}\sigma}+\min_{x_0\in\partial\Omega}\big( |\Bb\cdot\nf(x_0)|h-2\kappa (x_0)h^{2-\frac1\sigma}\big)-Ch^{6/5}\Big) |  u|^2\dd x\\
-(C_0+\bar C_0)h^{6/5}\|u\|^2\,.
\end{multline*}
The same argument also yields the following  inequality which is important for the localization properties of the eigenfunctions (see \cite[Thm.~5.2]{HK-tams}). There exist $h_0>0$ and $\tilde C $ such that, for all $h\in (0,h_0]$,
\begin{equation}\label{eq:qh-en-lb}
q_h(u)\geq \int_\Omega U_h(x) |u(x)|^2\dd x\,, 
\end{equation}
where
\begin{equation}\label{eq:potential}
U_h(x)=
\begin{cases}
-h^{2-\frac{2}\sigma}+ |\Bb\cdot\nf(p(x))|h-2\kappa (p(x))h^{2-\frac1\sigma}-\tilde Ch^{6/5}&~{\rm if~}{\rm dist}(x,\partial\Omega)<h^\frac 25 \\
0 &~{\rm if~}{\rm dist}(x,\partial\Omega)\geq h^\frac 25
\end{cases}
\,,
\end{equation}
with $p(x)\in\partial\Omega$ satisfies $|x-p(x)|={\rm dist}(x,\partial\Omega)$\,. 

\subsection{Upper bound of the principal eigenvalue}~\\
We choose an arbitrary point $x_0\in\partial\Omega$ and assume that its local $y$-coordinates is $y=0$.
We consider a test function of the form
\begin{equation}\label{eq:test-fct-tu}
\tilde u(y)=\tilde \chi(h^{-\rho}y) f(h^{-1/\sigma}y_3)\varphi_h(y_1)\exp\left(i\frac{ w(y_1,y_2)}{h}\right)\,,
\end{equation}
where $w$ is defined in  \eqref{defw}
\[f(\tau)=\sqrt{2}\,e^{-\tau}\,,\quad\tilde\chi(z)=\prod_{i=1}^3\chi(z_i)\quad(z=(z_1,z_2,z_3))\]
and $\chi\in C^\infty(\R)$ is a cut-off function such that $\chi=1$ on $[-\frac12,\frac12]$. We choose the parameter $\rho=\frac2{5}\in(0,\frac12)$ as in \eqref{eq:loc-qf**}.
The gauge function $w$ is introduced  in order to  ensure that \eqref{eq:A-lin-v1} holds. The function $\varphi_h$ is a ground state of the harmonic oscillator
\[ -h^2\frac{\dd^2}{\dd y_1^2}+( \mathcal B_3^0 y_1)^2\,,\]
and is given as follows\footnote{In the case $\mathcal B_3^0=0$, which amounts to $\Bb\cdot\n(x_0)=0$, the ground state energy  becomes $0$ and the \emph{generalized} $L^\infty$ ground state is a constant function. }
\[ \varphi_h(y_1)= \exp \left( -\frac{|\mathcal B_3^0|y_1^2}{2h}\right)\,.\]
In the Cartesian coordinates, it takes the form
\begin{equation}\label{eq:test-fct-u-cart}
u(x)= \exp\left( \frac{i\phi(x)}{h}\right) \tilde u(\Phi^{-1}(x) )\,,
\end{equation}
where $\Phi$ is the transformation that maps the Cartesian  coordinates to the boundary coordinates in a neighborhood of $x_0$ (see \eqref{eq:y0,W}), and $\phi$ is the gauge function required to assume that  $\tilde A_3=0$.

Thanks to Remark~\ref{rem:qf-ub}, we may write 
\[q_h(u)\leq \bar q_h^{\rm tran}(\tilde u)+(1+Ch^\rho) q_h^{\rm surf}(\tilde u)\,.\]
where the two auxilliary quadratic forms are defined in \eqref{eq.qtildetran} and \eqref{eq:loc-qf-lb2}. We also recall that $\alpha$ and $\sigma$ are related by \eqref{eq:alpha}.
The choice of $f$, its exponential decay, and the corresponding scaling give
\[ 
\bar q_h^{\rm tran}(u)\leq \int_{{ \{|y|<h^\rho,y_3>0\}}} \Big(-h^{2-\frac{2}\sigma}-2\kappa (x_0)h^{2-\frac1\sigma}+Ch^{2}\Big)|\tilde u|^2 |g|^{1/2}\dd y\,.
\]
Moreover, by using \eqref{eq.Taylor2}, and the classical inequality $|a+b|^2\leq(1+\varepsilon)|a|^2+(1+\varepsilon^{-1})|b|^2$ with $\varepsilon=h^\rho$, we get
\[q^{\rm surf}_h(\tilde u)\leq  (1+h^{\rho})\int_{{\{|y|<h^\rho,y_3>0\}}}|(-ih\nabla_{y_1,y_2}+\tilde \Ab^{\rm lin})v|^2\dd y+Ch^{3\rho}\int_{{\{|y|<h^\rho,y_3>0\}}}| v|^2\dd y\,.\]
where $\tilde A^{\rm lin}$ is defined in \eqref{eq:A-lin-fv}, and 
\[v(y_1,y_2,y_3)=\tilde\chi_h(y) f_h(y_3)\varphi_h(y_1)\,,\quad \chi_h(y)=\tilde \chi(h^{-\rho}y)\,,\quad f_h(y_3)=f(h^{-1/\sigma}y_3)\,. \]
Since $v$ is real-valued, we have
\[ |(-ih\nabla_{y_1,y_2}+\tilde \Ab^{\rm lin})v|^2 = h^2|\partial_{y_1}v|^2+h^2|\partial_{y_2}v|^2+\big( (\mathcal  B_2^0 y_3)^2 + ( \mathcal B_3^0 y_1-\mathcal B_1^0y_3)^2\big) |v|^2\,,\]
with
\[\partial_{y_1}v=\chi_h \,f_h\,\partial_{y_1}\varphi_h+ f_h\, \varphi_h\, \partial_{y_1}\chi_h\quad{\rm and}\quad
 \partial_{y_2}v=f_h\,\varphi_h\,\partial_{y_2}\chi_h\,.\]
 When $\mathcal  B_3^0 \neq 0$, by the exponential decay of $v$ in the $y_1$ direction  we get
\[ \int_{{\{|y|<h^\rho,y_3>0\}}} h^2|\partial_{y_1}v|^2 \dd y=
\int_{{\{|y|<h^\rho,y_3>0\}}} h^2|\chi_h f_h\partial_{y_1}\varphi_h|^2 \dd y +\mathcal O(h^\infty) \int_{{\{|y|<h^\rho,y_3>0\}}} |v|^2 \dd y\,.\]
When $\mathcal B_3^0=0$, $\varphi_h$ is constant, hence  $\chi_h \,f_h\,\partial_{y_1}\varphi_h=0$ and
 \[\int_{{\{|y|<h^\rho,y_3>0\}}} h^2|\partial_{y_1}v|^2 \dd y=\mathcal O(h^{2-2\rho}) \int_{{\{|y|<h^\rho,y_3>0\}}} |v|^2 \dd y\,.\]
 Hence, in each case, we have 
 \[ \int_{{\{|y|<h^\rho,y_3>0\}}} h^2|\partial_{y_1}v|^2 \dd y=
\int_{{\{|y|<h^\rho,y_3>0\}}} h^2|\chi_h f_h\partial_{y_1}\varphi_h|^2 \dd y +\mathcal O(h^{2-2\rho}) \int_{{\{|y|<h^\rho,y_3>0\}}} |v|^2 \dd y\,.\]
We also have the estimate
\[\int_{{\{|y|<h^\rho,y_3>0\}}} h^2|\partial_{y_2}v|^2\dd y= \mathcal O(h^{2-2\rho})  \int_{{\{|y|<h^\rho,y_3>0\}}}|v|^2\dd y\,.\]
Moreover,
\begin{align*}
\int_{{\{|y|<h^\rho,y_3>0\}}}&\big( (\mathcal  B_2^0 y_3)^2 + ( \mathcal B_3^0 y_1-\mathcal B_1^0y_3)^2\big) |v|^2\dd y\\
&=
\int_{{\{|y|<h^\rho,y_3>0\}}}\big( ( \mathcal B_3^0)^2 y_1^2 + \mathcal O(y_1y_3)+\mathcal O(y_3^2)\big) |v|^2\dd y\\
&=\int_{{\{|y|<h^\rho,y_3>0\}}}  ( \mathcal B_3^0)^2 y_1^2 \chi_h^2 f_h^2|\varphi_h|^2 \dd y +
\big(\mathcal O(h^{\frac2\sigma})+\mathcal O(h^{\frac12+\frac1\sigma})\big)\int_{{\{|y|<h^\rho,y_3>0\}}} |v|^2\dd y\,.
\end{align*}
Collecting the foregoing estimates we get, for some constant $C>0$, 
\begin{align*}
\int_{{\{|y|<h^\rho,y_3>0\}}}&|(-ih\nabla_{y_1,y_2}+\tilde \Ab^{\rm lin})v|^2\dd y\\&=
\int_{{\{|y|<h^\rho,y_3>0\}}}\tilde\chi_h^2\left(|(-ih\partial_1)(f_h\varphi_h)|^2+|(\tilde B_3^0 y_1)(f_h\varphi_h)|^2\right) \dd y\\
&\text{ }\quad\qquad \qquad  + C 
\Big( h^{\frac2\sigma}+h^{\frac 12+\frac 1\sigma}+h^{2-2\rho}  \Big) \|v\|^2_{ L^2(|g|^{\frac 12}\dd y)}\\
&\leq \big((|\mathcal  B^0_3| h+C(h^{\frac2\sigma} +h^{\frac 12+\frac 1\sigma}+h^{2-2\rho}) \big)\|v\|^2_{ L^2(|g|^{\frac 12}\dd y)}\,.
\end{align*}
Therefore,
\begin{multline}\label{eq:up-qf-test-fc-u}
q_h(u)\leq \int_{{\{|y|<h^\rho,y_3>0\}}} \Big(-h^{2-\frac{2}\sigma}-2\kappa (x_0)h^{2-\frac1\sigma}+Ch^{2}\Big)|\tilde u|^2 |g|^{1/2}\dd y\\
+\Big(|\mathcal  B_3^0|h+ r(h;x_0)\Big)\int_{{\{|y|<h^\rho,y_3>0\}}}|\tilde u|^2|g|^{1/2}\dd y\,,
\end{multline}
where $r(h;x_0)=o(h)$ uniformly with respect to $x_0$ (due to our conditions on $\sigma$ and $\rho$). \\
For $\sigma=1$, we have $r(h;x_0)=\mathcal O(h^{6/5})$.

The min-max principle now yields the upper bound
\begin{equation}\label{uppbnd}
\mu(h,\Bb)\leq -h^{2-\frac{2}\sigma}-2\kappa (x_0)h^{2-\frac1\sigma}+ |\Bb\cdot\nf(x_0)|\, h+ o(h)\,,
\end{equation}
{with $o(h)$ being uniformly controlled with respect to $x_0$ (by compactness of $\partial\Omega$).}  Minimizing over $x_0\in\partial\Omega$, we get
\[\mu(h,\Bb)\leq -h^{2-\frac{2}\sigma}+\min_{x_0\in\partial\Omega}\big( |\Bb\cdot\nf(x_0)|h-2\kappa (x_0)h^{2-\frac1\sigma}\big)+ o(h)\,.\]
This concludes the proof of Theorem~\ref{thm:ev}, in light of \eqref{eq:ev-g-h}.

\begin{rem}\label{rem:rem-c-r}
Thanks to \eqref{eq:up-qf-test-fc-u}, the remainder term in \eqref{uppbnd} becomes $\mathcal O(h^{6/5})$ in the case when $\sigma=1$. Consequently, in this case, the eigenvalue asymptotics reads as follows
\[ \mu(h,\Bb)= -h^{2-\frac{2}\sigma}+\min_{x_0\in\partial\Omega}\big( |\Bb\cdot\nf(x_0)|h-2\kappa (x_0)h^{2-\frac1\sigma}\big)+\mathcal O(h^{6/5})\,.\]
The improved remainder term will be helpful in the analysis of the ball situation in Sec.~\ref{sec:ball}. 
\end{rem}

\subsection{Upper bound for the $n$-th eigenvalue}
For every positive integer $n$, the estimate in \eqref{eq:up-qf-test-fc-u} still holds when the functions $\tilde u$ in \eqref{eq:test-fct-tu} and $u$ in \eqref{eq:test-fct-u-cart} are replaced by the functions $\tilde u_n$ and $u_n$ defined as follows:
\[\tilde u_n(y)= \chi(h^{-\rho}y_1)\chi(h^{-\rho}y_3) \vartheta_n(y_2)  f(h^{-1/\sigma}y_3)\varphi_h(y_1)\exp\left(i\frac{ w(y_1,y_2)}{h}\right)\,, \]
\[u_n(x)=\exp\left( \frac{i\phi(x)}{h}\right) \tilde u_n(\Phi^{-1}(x) )\,,\]
and
\[\vartheta_n(y_2)=\mathbf 1_{(-h^\rho,h^\rho)}(y_2)\sin\left(\frac{n\pi(y_2-h^\rho)}{h^\rho}\right)\,.\]
The functions $(\vartheta_n)_{n\geq 1}$ are orthogonal\,\footnote{The functions $\vartheta_n(y_2)$ are in fact the eigenfunctions of the Dirichlet $1$D Laplace operator on $[-h^\rho,h^\rho]$}. This ensures that the space  $M_n={\rm Span}(u_1,\cdots,u_n)$ satisfies ${\rm dim}(M_n)=n$. The min-max principle yields that, with a remainder uniform in $x_0$, we have\,\footnote{A special attention is needed for the case when $\mathcal B_3^0=0$, which we handle in the same way done along the proof of \eqref{uppbnd}.}
\[\mu_n(h,\Bb)\leq \max_{u\in M_n} \frac{q_h(u)}{\|u \|^2}\leq  -1+ \big(|\Bb\cdot\nf(x_0)|h-2\kappa (x_0)h^{2-\frac1\sigma}\big)+  o(h)\,,\]
where $\mu_n(h,\Bb)$ denotes the $n$'th eigenvalue counting multiplicities. Minimizing over $x_0\in\partial\Omega)$, we get 
\begin{equation}\label{eq:ub-ev-n}
\mu_n(h,\Bb)\leq -h^{2-\frac{2}\sigma}+\min_{x_0\in\partial\Omega}\big( |\Bb\cdot\nf(x_0)|h-2\kappa (x_0)h^{2-\frac1\sigma}\big)+  o(h)\,.
\end{equation}

\section{Effective boundary operator in the critical regime}\label{sec:efop}

\subsection{Preliminaries}
We assume that $\sigma=1$  in \eqref{thm:ev} (hence $\alpha=1$ in \eqref{eq:qh}).  The quadratic form in \eqref{eq:qh} is then
\begin{equation}\label{eq:qh-c} 
q_h(u)=\int_\Omega|(-ih\nabla+\Ab)u|^2\dd x-h\int_{\partial\Omega}|u|^2\dd s(x)\,.
\end{equation}
This regime  is critical since the contribution of the magnetic field and the Robin parameter are of the same order. In the semiclassical version, our estimate reads as follows (see Remark \ref{rem:rem-c-r})
\begin{equation}\label{eq:ev-critical}
\mu(h,\Bb)= -1+\min_{x_*\in\partial\Omega}\big( |\Bb\cdot\nf(x_*)|-2\kappa (x_*)\big)h+\mathcal O(h^{6/5})\,.
\end{equation}
Observing that $\mu_n(h,\Bb) \geq \mu_1(h,\Bb)$ and \eqref{eq:ub-ev-n}, the expansion  in \eqref{eq:ev-critical} continues to hold for the  $n$th  eigenvalue $\mu_n(h,\Bb)$ (with $n$ fixed), namely,
\begin{equation}\label{eq:ev-critical-n}
\mu_n(h,\Bb)= -1+\min_{x_*\in\partial\Omega}\big( |\Bb\cdot\nf(x_*)|-2\kappa (x_*)\big)h+\mathcal O_n(h^{6/5})\,.
\end{equation} 
By a standard argument (see \cite[Thm.~5.1]{HK-tams}), for any $n\in \mathbb N^*$, there exist $h_n>0$ and $C_n>0$, such that for $h\in (0,h_n]$, any $n$-th $L^2$-normalized eigenfunction $u_n$,  is  localized near  the boundary as follows
\begin{equation}\label{eq:dec-ef}
 \int_\Omega\left(|u_n|^2+|(-ih\nabla+\Ab) u_n|^2\right)\exp\left( \frac{{\rm dist}(x,\partial\Omega)}{4h}\right)\dd x\leq C_n\,. 
 \end{equation}
The formula in \eqref{eq:ev-critical}, along with the one in \eqref{eq:potential} and Agmon estimates,  allows us to refine the localization of the $n$-th  eigenfunction  near the set (see \cite[Sec.~8.2.3]{HM3d})
\begin{equation}\label{eq:set-S}
S:=\Big\{x\in\partial\Omega,~ |\Bb\cdot\nf(x)|-2\kappa (x)=\min_{x_*\in\partial\Omega}\big( |\Bb\cdot\nf(x_*)|-2\kappa (x_*)\big)\Big\}\,.
\end{equation}
More precisely, we have  Proposition~\ref{prop:dec-c-r} below. Its statement involves a smooth function $\chi:\R\to[0,1]$  supported in $[-2\epsilon_0,2 \epsilon_0]$ such that $\chi=1$ on $[-\epsilon_0,\epsilon_0]$, where $\epsilon_0$ is small enough. We also need the potential function $V$ defined in a neighborhood of $\partial\Omega$ as follows
\[V(x)= |\Bb\cdot\nf(p(x))|-2\kappa (p(x))\quad{\rm and}\quad E=\min_{x\in\partial\Omega} V(x)\,,\]
where $p(x)\in\partial\Omega$ is given by ${\rm dist}(x,p(x))={\rm dist}(x,\partial\Omega)$.  We also denote by $d_{V-E} (\cdot ,S)$ the Agmon distance to $S$ in $\partial \Omega$ associated with the potential $(V-E)$ (see \cite[Sec.~3.2, p.~19]{H-ln}).
\begin{proposition}\label{prop:dec-c-r}
Given $\tau \in  (0,1)$,  $n\geq 1$ and any $\tilde \eta$,  there exist positive constants $h_n,C_n, \delta(\tilde\eta)$ such that, $\lim_{\tilde \eta\rightarrow 0} \delta(\tilde\eta)=0$ and such that, for all $h\in(0,h_n]$, the following estimate holds
\begin{equation}\label{eq:dec-agm-tan} 
\int_{\Omega}\left(|u_n|^2+h |(-i\nabla+\Ab) u_n|^2\right)\exp\left( 2 \tau \chi ({\rm dist }(x,\partial \Omega)) \frac{ \phi (x) }{h^{1/2}} \right)\dd x \leq C_{n} \exp \delta(\tilde \eta) h^{-\frac 12} \,,
\end{equation}
where 
$$
\phi(x):=d_{V-E} (p(x),S) \,.
$$
In particular, for each $\varepsilon\in(0,1)$, there exists $C_\epsilon >0$ and $h_\epsilon$ such that, for $h\in (0,h_\epsilon]$,  
\begin{equation}\label{eq:dec-set-S} 
 \int_{\Omega\setminus S_{\varepsilon}}\left(|u_n|^2+h |(-i\nabla+\Ab) u_n|^2\right)\dd x\leq    \exp - C_\epsilon h^{-\frac 12} \,\,,
\end{equation}
where 
\begin{equation}\label{eq:S-h-epsilon}
S_{\varepsilon}=\Big\{x\in\Omega,~{\rm dist}(x,\partial\Omega)<\varepsilon~\&~ |\Bb\cdot\nf(p(x))|-2\kappa (p(x))< \min\limits_{x_*\in\partial\Omega}\big( |\Bb\cdot\nf(x_*)|-2\kappa (x_*)\big)+\varepsilon\Big\}.
\end{equation}
\end{proposition}
\begin{proof}
The estimate in \eqref{eq:dec-set-S} results from \eqref{eq:dec-agm-tan} and \eqref{eq:dec-ef} by a clever choice of $\tilde \eta$, noting also  that there exists $c_0>0$ such that (see \cite[Lem.~3.2.1, p.~20]{H-ln})
$$
c_0 (V(x)-E)^{\frac 32} \leq \phi(x)\,.
$$
So we need to understand the decay property close to the boundary.  
Consider the function
\[\Phi(x)= \exp\left( \tau h^{-1/2} \chi ({\rm dist}(x,\partial\Omega)) \phi\big(x)\right)\,.\]
We note that  $p$  is well defined on the support of $\chi\left({\rm dist}(x,\partial\Omega)\right)$ and that, by Remark~\ref{rem:qf-b-c}, 
there exists a positive constant $M$  such that  $$|\nabla\phi|^2\leq \big(1+M{\rm dist}(x,\partial\Omega)\big)(V-E)\mbox{  a.e. on } \{{\rm dist}(x,\partial\Omega) < 2\epsilon_0\}\,.$$
We write the identity
\begin{equation}\label{eq.Agmonphia}
q_h(\Phi u_n)-h^2\int_\Omega |\nabla \Phi |^2|u_n|^2\,dx=\mu_h(h)\| \Phi   u_n\|_{L^2(\Omega)}^2\,,
\end{equation}
where $q_h$ is introduced in \eqref{eq:qh-c}.
Thanks to \eqref{eq:qh-en-lb} and \eqref{eq:ev-critical-n}, we get
\begin{multline*}
\int_{{\rm dist}(x,\partial\Omega)<h^{2/5}}\left(U_h|\Phi u_n|^2-h^2|\nabla\Phi|^2-(-1+hE+Ch^{6/5})|\Phi u_n|^2\right)\dd x\\
\leq C\int_{{\rm dist}(x,\partial\Omega)>h^{2/5}}(h^2|\nabla\Phi|^2+|\Phi|^2) |u_n|^2\dd x\,,
\end{multline*}
and
\begin{multline*}
\int_{{\rm dist}(x,\partial\Omega)<h^{2/5}}\left((hV-hE-Ch^{6/5})|\Phi u_n|^2-h^2|\nabla\Phi|^2\right)\dd x\\
\leq C\int_{{\rm dist}(x,\partial\Omega)>h^{2/5}}(h^2|\nabla\Phi|^2+|\Phi|^2) |u_n|^2\dd x\,.
\end{multline*}
Notice that
\[|\nabla\Phi|^2=h^{-1}\tau^2\, \Phi^2\left|\phi(x)\chi'\left(\mathrm{dist}(x,\partial\Omega)\right)\nabla\mathrm{dist}(x,\partial\Omega)+\chi\left(\mathrm{dist}(x,\partial\Omega)\right)\nabla\phi(x)\right|^2\,.\]
Thus,
\begin{equation}\label{eq.nablaphia}
|\nabla\Phi|^2\leq  h^{-1} (\tau^2 + \eta) \Phi^2\, \left(|\nabla\phi(x)|^2+C_\eta |\chi' (\mathrm{dist}(x,\partial\Omega))|\right) \,.
\end{equation}
Using \eqref{eq.Agmonphia} and \eqref{eq.nablaphia}, we deduce that 
\begin{multline*}
\int_{{\rm dist}(x,\partial\Omega)<h^{2/5}} \Big ( hV-hE- h (\tau^2+\eta) |\nabla\phi|^2-Ch^{6/5}\Big) |\Phi  u_n|^2\dd x\\
\leq  C_\eta h \int |\chi'\left(\mathrm{dist}(x,\partial\Omega) \right||\Phi u_n|^2\dd x+C\int_{{\rm dist}(x,\partial\Omega)>h^{2/5}}(h^2|\nabla\Phi|^2+|\Phi|^2) |u_n|^2\dd x  \,.
\end{multline*}
Thanks to \eqref{eq:dec-ef}, we get
\begin{equation*}
\int_{{\rm dist}(x,\partial\Omega)<h^{2/5}} \Big ( hV-hE- h (\tau^2+\eta) |\nabla\phi|^2-Ch^{6/5}\Big) |\Phi  u_n|^2\dd x\\
\leq  Ce^{-c/h^{3/5}}  \,.
\end{equation*}
Now, we choose $ \eta = \frac{1-\tau^2}2$.
Thus, 
\begin{equation}\label{eq:dec-c-r*}
\int_{{\rm dist}(x,\partial\Omega)<h^{2/5}} \Big ( \left(\frac{(1-\tau^2)(1-Mh^{2/5})}2)\right)(V-E)-Ch^{1/5}\Big) |\Phi  u_n|^2\dd x
\leq  \frac{C}{h} e^{-c/h^{3/5}}  \,.
\end{equation}
For any $\tilde \eta >0$, we get
\begin{equation}\label{eq:dec-c-r*new}
\begin{array}{l}
\int_{\{{\rm dist}(x,\partial\Omega)<h^{2/5}\} \cap \{V(x)- E >\tilde \eta\} } \Big ( \frac{1-\tau^2}{4} (V-E)-Ch^{1/5}\Big) |\Phi  u_n|^2\dd x\\
\qquad\qquad\qquad \leq  \frac{C_{\tilde \eta}}{h} e^{-c/h^{3/5}}  + \tilde C \int_{ V(x)-E <  \tilde \eta }  |\Phi  u_n|^2 dx    \,.
\end{array}
\end{equation}
So we infer from \eqref{eq:dec-c-r*new} that for any $\tilde \eta>0$,  there exists $\hat C_{\tilde \eta }  >0$ such that
\[
\int_{{\rm dist}(x,\partial\Omega)<h^{2/5}} |\Phi  u_n|^2\dd x
\leq  \hat C_{\tilde \eta}       \exp \delta(\tilde \eta)  h^{-\frac 12}\,.
\]
where $\delta(\tilde \eta) \rightarrow 0$ as $\tilde \eta \rightarrow 0$.\\
Implementing again, \eqref{eq:dec-ef}, we have proven that for any $\tilde \eta$, there exists $\hat C_{\tilde \eta} >0$ and $h_{\tilde \eta} >0$ such that, for $h\in (0,h_{\tilde \eta})$, 
\begin{equation}
\int_{\Omega} |\Phi  u_n|^2\dd x
\leq  2 \hat C_{\tilde \eta}       \exp \delta(\tilde \eta) h^{-\frac 12}
\,.
\end{equation}
Inserting this into \eqref{eq.Agmonphia}, we eventually  get the decay estimate, close to the boundary.
\end{proof}

\begin{rem}
When $S=\{x_0\}$ and $V$ has a non degenerate unique minimum at $x_0$, we can take $\tilde \eta  = A h^\frac 15$ and get  $\delta(\tilde \eta) \sim B h^{\frac 15}$\,.
\end{rem}

\subsection{Reduction to an operator near $x_0$}

In light of the estimates in \eqref{eq:dec-ef} and \eqref{eq:dec-set-S},  it is sufficient to analyze the quadratic form in \eqref{eq:qh-c} on functions  supported in $\{{\rm dist}(x,\partial\Omega)<h^\varrho\}\cap S_{\varepsilon}$, with $\varrho,\varepsilon\in (0,1)$. We explain this below. We recall that, under our assumptions in Theorem \ref{thm:ev-c},
\begin{equation}\label{eq:S-M0} 
S= \{ x_0\}\,.
\end{equation}
Choose $\delta\in(0,1)$, an open subset $D$ of $\R^2$, with a smooth boundary, and boundary coordinates $y:=(y',y_3)\in \mathcal{V}=D\times (0,\delta)$ that maps $\mathcal V$ to a neighborhood of $\mathcal N_0$ of  the point $M_0$. Recall that $y_3$ denotes the distance to the boundary, and the coordinates of $x_0$  are defined by $y=0$. 

If we consider the operator defined by the restriction of the quadratic form in \eqref{eq:qh} on functions $u\in H^1(\mathcal N_0)$ satisfying $u=0$ on $\Omega\cap\partial \mathcal N_0$, we end up with an operator $\mathcal L_h^{\mathcal N_0}$ whose $n$-th eigenvalue satisfies
\begin{equation}\label{eq.locbnd}
\mu_n(h,\Bb)\leq \mu_n(\mathcal L_h^{\mathcal N_0})\leq  \mu_n(h,\Bb)+\mathcal O(h^\infty)\,.
\end{equation}
The space $L^2(\mathcal N_0)$ is transformed, after passing to the boundary coordinates, to  the space $L^2(\mathcal V,\dd m)$ with the \emph{weighted} measure $dm=|g(y',y_3)|^{1/2}\dd y$. We introduce also the  spaces $L^2(D)$ and $L^2(D,\dd s)$,   with the canonical  measure $dy'$ and  weighted measure,
\[\dd s=|G(y')|^{1/2}\dd y'=|g(y',0)|^{1/2}\dd y'\,,\]
respectively. Note that  $L^2(D,\dd s)$ is the transform of the space $ L^2(\partial\Omega\cap\mathcal N_0)$ by the boundary coordinates. In these coordinates ((see \eqref{eqg0}-\eqref{eq:g,ij*},  \eqref{eq:A3=0}, and \eqref{defqu})), the quadratic form of the operator $\mathcal L_h^{\mathcal N_0}$ is
\begin{align*}
&q_h(u;\mathcal V)\\
&\quad =\int_{\mathcal{V}} |(-ih\partial_{y_3}-\tilde A_3(y',y_3)) \tilde u |^2 |g(y',y_3)|^{1/2}\dd y_3\dd y'-h\int_{D}|\tilde u(y',0)|^2|g(y',0)|^{1/2}\dd y'\\
&\qquad +\int_{\mathcal{V}}\sum_{k,\ell \in \{1,2\}}  g^{k\ell}(y',y_3) ((-ih\partial_k-\tilde A_k (y',y_3)) \tilde u )\,\overline{((-ih\partial_\ell-\tilde A_\ell(y',y_3))\tilde u )} \, |g(y',y_3)|^{1/2}\dd y'\dd y_3\,.
\end{align*} 
Up to a change of gauge, we may assume that $\tilde A_3=0$.

We will derive then a `local' effective unbounded  operator in  the weighted space $L^2(D)$.

\subsection{The effective operator}
\subsubsection{Rescaling and splitting of the quadratic form}
We recall that
\begin{align*}
&q_h(u;\mathcal V)\\
&\quad =\int_{\mathcal{V}} h^2|\partial_{y_3} \tilde u |^2 |g(y',y_3)|^{1/2}\dd y_3\dd y'-h\int_{D}|\tilde u(y',0)|^2|g(y',0)|^{1/2}\dd y'\\
&\qquad +\int_{\mathcal{V}}\sum_{k,\ell \in \{1,2\}}  g^{k\ell}(y',y_3) ((-ih\partial_k-\tilde A_k (y',y_3)) \tilde u )\,\overline{((-ih\partial_\ell-\tilde A_\ell(y',y_3))\tilde u )} \, |g(y',y_3)|^{1/2}\dd y'\dd y_3\,.
\end{align*} 
Introducing the rescaled normal variable $t=h^{-1}y_3$, the function $\tilde u$ is to transformed to the new function $\psi(y',t):=\tilde u(y',ht)$ and the domain $\mathcal V$ is transformed to 
\begin{equation}\label{eq:Vh-c}
\mathcal V_h=D\times \Big(0,\frac{\delta}h\Big)\,.
\end{equation} 
We obtain then the new quadratic form, and the new $L^2$-norm:
\begin{equation}\label{defQ}
\|u\|^2=h\|\psi\|^2\,,\quad q_h(u;\mathcal V)=h Q_h(\psi)\,,\quad Q_h(\psi):=Q(\psi):=Q^{\mathrm{tr}}(\psi)+Q^{\mathrm{bnd} }(\psi)\,,
\end{equation}
where
\[Q^{\mathrm{tr}}(\psi)=\int_{D}\int_0^{\delta/h} \Big( |\partial_{t}\psi|^2 |g(y',h t)|^{1/2}\dd t -|\psi(y',0)|^2|g(y',0)|^{1/2}\Big)  \dd y'\,,\]
and
\begin{multline*}
Q^{\mathrm{bnd} }(\psi)\\
=\int_{D} \int_0^{\delta/h}\sum_{k,\ell \in \{1,2\}}g^{k\ell}(y',h t)(-ih\partial_k-\tilde A_k(y',h t))\psi\overline{(-ih\partial_\ell-\tilde A_\ell(y',h t))\psi}|g(y',h t)|^{1/2}\dd t\dd y'\,.
\end{multline*}
The elements of the form domain satisfy
\[\psi\in H^1(D\times (0,\delta/h)),\quad~\psi=0{~\rm on~} (\partial D)\times (0,\delta/h)~ {\rm and~on~}D\times\{\delta/h\}\,.\]
The operator associated with $Q_h$ is denoted by $\mathscr{L}_h$, and its eigenvalues are denoted by $(\mu_n(h))_{n\geq 1}$.

\subsubsection{On the transverse operator}
Before defining our effective operator, one needs to introduce the following partial transverse quadratic form   
\[ f \mapsto q_{{h} , y'}(f)=\int_0^{\delta/h}| f '(t) |^2 |g(y',h t)|^{1/2}\dd t-|f(0)|^2|g(y',0)|^{1/2}\,,\]
in the ambient Hilbert space, $L^2((0,\delta/h),|g(y',h t)|^{1/2}\dd t)$, and defined on the form domain 
\[
\mathcal D (q_{{h} , y'}):= \{f~:~f,f'\in L^2((0,\delta/h),|g(y',h t)|^{1/2}\dd t)~{\rm and~}f(\delta/h)=0\}\,.\]
We denote by $\mu(h, y')$ the groundstate of the associated operator and by $f_{h, y'}$ the corresponding \emph{positive} and normalized (in $L^2((0,\delta/h),|g(y',h t)|^{1/2}\dd t)$) eigenfunction. Note that these depend smoothly on the variable $y'$,  by standard perturbation theory. We may prove, as in \cite[Sections 2.3 \& 7.2, with $T=\frac{\delta}{h}$, $B=h$]{HKR17}, that 
\begin{equation}\label{eq.muy'}
\mu(h,y')=-1-2h\kappa(y')+\mu^{[2]}(y')h^2+\mathcal{O}(h^3)\,,
\end{equation}
and, in the $L^2$-sense,
\begin{equation}\label{eq.deriv-f}
\partial_{y_k} f_{h,y'}=\mathcal{O}(h)\,.
\end{equation}
\subsubsection{Description of the effective operator}
Our effective operator is the self-adjoint operator, in the space $L^2(D)$, with domain $H^2(D)\cap H^1_0(D)$, and defined as follows
\begin{equation}\label{eq:eff-op-2D}
\mathscr{L}^{\mathrm{eff}}=\sum_{k\ell} (P_\ell \alpha_{k\ell}P_k+\beta_{k\ell} P_k+P_k { \beta_{k\ell}}+\gamma_{k\ell})+\mu(y',{h})-h^2\rho(y',h) \,,\end{equation}
where 
\begin{equation}\label{eq:Pk-op}
P_k=-i{h} \partial_k-\tilde A_k^0\,,\quad \tilde A_k^0(y')=\tilde A_k(y',0)\,,
\end{equation}
\begin{equation}\label{eq:coef-abg}
\begin{split}
	\alpha_{k\ell}&=\int_{0}^{\delta/{h} }f_{{h} ,y'}^2(t)g^{k\ell}(y',{h}  t)|g(y',{h}  t)|^{\frac 12}\dd t\\
	\beta_{k\ell}&=\int_{0}^{\delta/{h} }f_{{h} ,y'}^2(t)g^{k\ell}(y',{h}  t)|g(y',{h}  t)|^{\frac 12}(\tilde A^0_\ell-\tilde A_\ell)\dd t\\
	\gamma_{k\ell}&=\int_{0}^{\delta/{h} }f_{{h} ,y'}^2(t)g^{k\ell}(y',{h}  t)|g(y',{h}  t)|^{\frac 12}(\tilde A_k^0-\tilde A_k)(\tilde A_\ell^0-\tilde A_\ell)\dd t\,,
	\end{split}
	\end{equation}
	and 
	\begin{equation}\label{eq:rho-eff-op}
	\rho(y',h)=\sum_{k\ell}\partial_\ell\left(\int_0^{\delta/h}g^{k\ell}(y',{h}  t)f_{h,y'}\partial_kf_{h,y'}|g(y',{h}  t)|^{\frac 12}\dd t\right)\,.
	\end{equation}
	The coefficients $\alpha_{kl},\beta_{k\ell},\gamma_{k\ell}$  depend on ${h} $ and $y'$ only. Note that $(\alpha_{k\ell})$ and $(\gamma_{k\ell})$ are symmetric. 
	\begin{rem}\label{rem.abcr}
	We may notice that, due to the exponential decay of $f_{h,y'}$, we have, uniformly in $y'\in D$,
	\begin{equation}\label{eqinrema}
	\alpha=\alpha^{[0]}+h\alpha^{[1]}+\mathcal{O}(h^2)\,,\quad \beta=h\beta^{[1]}+\mathcal{O}(h^2)\,,\quad \gamma=h^2\gamma^{[2]}+\mathcal{O}(h^3)\,,
	\end{equation}
	and
\begin{equation}\label{eqinremb}
\rho=\mathcal{O}(h)\,.
\end{equation}
	\end{rem}

\subsection{Reduction to an effective operator}
The aim of this section is to prove the following proposition, whose proof is inspired by \cite{KKR16}.

\begin{proposition}\label{prop.eff}
For all $n\geq 1$, there exist $h_0>0$ and $C>0$ such that, for all $h\in(0,h_0]$,
\[|\mu_n(h)-\mu^{\mathrm{eff}}_n(h)|\leq Ch^3 \,.\]	
\end{proposition}

\subsubsection{Upper bound}
\begin{lemma}
Consider
\[\psi(y',t)=f_{{h} ,y'}(t)\varphi(y')\,,\]
with $\varphi\in H^1_0(D)$. We write
\[\begin{split}
Q(\psi)=\int_D\mu(y',{h} )|\varphi(y')|^2\dd y'+Q^{\mathrm{tg}}(\varphi) +E_h(\varphi)
\end{split}
\,,\]
where
\begin{equation} \label{defQt}
\begin{split}&Q^{\mathrm{tg}}(\varphi)=\\
&\int_{\mathcal{V}_{h} }f_{{h} ,y'}^2\sum_{k\ell}g^{k\ell}(y',{h}  t)(-i{h} \partial_k-\tilde A_k(y',{h}  t))\varphi\overline{(-i{h} \partial_\ell-\tilde A_\ell(y',{h}  t))\varphi}|g(y',{h}  t)|^{\frac 12}\dd y'\dd t\,,\end{split}
\end{equation}
and  $\mathcal{V}_{h}$ is  introduced in \eqref{eq:Vh-c}.
Then the term $E_h(\varphi)$ satisfies
\[| E_h(\varphi)-E_h^0(\varphi)|\leq Ch^3\|\varphi\|^2\,,\quad E_h^0(\varphi)=-h^2\left\langle\rho(y',h)\varphi,\varphi\right\rangle\,.\]
where $\rho(y',h)$ is introduced in \eqref{eq:rho-eff-op}.
\end{lemma}
\begin{proof}
The term $E_h(\varphi)$ comes from the fact that $f_{h,y'}$ depends on $y'$. We have
\[\begin{split}&E_h(\varphi)=\\
&\int_{\mathcal{V}_{h} }\sum_{k\ell}g^{k\ell}(y',{h}  t)[(-i{h} \partial_k-\tilde A_k(y',{h}  t)),f_{h,y'}]\varphi\overline{(-i{h} \partial_\ell-\tilde A_\ell(y',{h}  t))\psi}|g(y',{h}  t)|^{\frac 12}\dd y'\dd t\\
&+\int_{\mathcal{V}_{h} }\sum_{k\ell}g^{k\ell}(y',{h}  t)f_{h,y'}(t)(-i{h} \partial_k-\tilde A_k(y',{h}  t))\varphi\overline{[(-i{h} \partial_\ell-\tilde A_\ell(y',{h}  t)), f_{h,y'}]\varphi}|g(y',{h}  t)|^{\frac 12}\dd y'\dd t\,.
\end{split}\]
Let us first estimate the error term $E_h(\varphi)$. We have
\[\begin{split}&E_h(\varphi)=\\
&-ih\int_{\mathcal{V}_{h} }\sum_{k\ell}g^{k\ell}(y',{h}  t)\partial_kf_{h,y'}\varphi\overline{(-i{h} \partial_\ell-\tilde A_\ell(y',{h}  t))\psi}|g(y',{h}  t)|^{\frac 12}\dd y'\dd t\\
&+ih\int_{\mathcal{V}_{h} }\sum_{k\ell}g^{k\ell}(y',{h}  t)f_{h,y'}(t)(-i{h} \partial_k-\tilde A_k(y',{h}  t))\varphi\partial_\ell f_{h,y'}\overline{\varphi}|g(y',{h}  t)|^{\frac 12}\dd y'\dd t\,,
\end{split}\]
and then 
\[\begin{split}&E_h(\varphi)=\\
&-ih\int_{\mathcal{V}_{h} }\sum_{k\ell}g^{k\ell}(y',{h}  t)f_{h,y'}\partial_kf_{h,y'}\varphi\overline{(-i{h} \partial_\ell-\tilde A_\ell(y',{h}  t))\varphi}|g(y',{h}  t)|^{\frac 12}\dd y'\dd t\\
&+h^2\int_{\mathcal{V}_{h} }\sum_{k\ell}g^{k\ell}(y',{h}  t)\partial_kf_{h,y'}\partial_\ell f_{h,y'}|\varphi|^2|g(y',{h}  t)|^{\frac 12}\dd y'\dd t\\
&+ih\int_{\mathcal{V}_{h} }\sum_{k\ell}g^{k\ell}(y',{h}  t)f_{h,y'}(t)(-i{h} \partial_k-\tilde A_k(y',{h}  t))\varphi\partial_\ell f_{h,y'}\overline{\varphi}|g(y',{h}  t)|^{\frac 12}\dd y'\dd t\,.
\end{split}\]
Let us now replace $\tilde A_k(y',{h}  t)$ by $\tilde A_k^0(y')$.
We get
\[\begin{split}&E_h(\varphi)=\\
&-ih\int_{\mathcal{V}_{h} }\sum_{k\ell}g^{k\ell}(y',{h}  t)f_{h,y'}\partial_kf_{h,y'}\varphi\overline{(-i{h} \partial_\ell-\tilde A^0_\ell(y')\varphi}|g(y',{h}  t)|^{\frac 12}\dd y'\dd t\\
&+h^2\int_{\mathcal{V}_{h} }\sum_{k\ell}g^{k\ell}(y',{h}  t)\partial_kf_{h,y'}\partial_\ell f_{h,y'}|\varphi|^2|g(y',{h}  t)|^{\frac 12}\dd y'\dd t\\
&+ih\int_{\mathcal{V}_{h} }\sum_{k\ell}g^{k\ell}(y',{h}  t)f_{h,y'}(t)(-i{h} \partial_k-\tilde A^0_k(y')\varphi\partial_\ell f_{h,y'}\overline{\varphi}|g(y',{h}  t)|^{\frac 12}\dd y'\dd t\\
&-ih\int_{\mathcal{V}_{h} }\sum_{k\ell}g^{k\ell}(y',{h}  t)f_{h,y'}\partial_kf_{h,y'}\varphi\overline{(\tilde A^0_\ell(y')-\tilde A_k(y',ht))\varphi}|g(y',{h}  t)|^{\frac 12}\dd y'\dd t\\
&+ih\int_{\mathcal{V}_{h} }\sum_{k\ell}g^{k\ell}(y',{h}  t)f_{h,y'}(t)(\tilde A^0_k(y')-\tilde A_k(y',ht))\varphi\partial_\ell f_{h,y'}\overline{\varphi}|g(y',{h}  t)|^{\frac 12}\dd y'\dd t\,.
\end{split}\]
 We recall   from \eqref{eq.deriv-f}  that $\partial_k f_{h,y'}=\mathcal{O}(h)$.
  Remembering the definition of  the operators $P_k$ introduced in \eqref{eq:Pk-op},  we have
\[| E_h(\varphi)-E_h^0(\varphi)|\leq Ch^3\|\varphi\|^2\,,\]
with
\[E_h^0(\varphi)=\sum_{k\ell}\int_D \left[P_k\varphi\,\overline{\tilde\beta_{\ell k}\varphi}+\tilde\beta_{k\ell}\varphi \overline{P_\ell\varphi}\right]  \dd y'=\sum_{k\ell}\langle \left(\overline{\tilde\beta_{\ell k}}P_k+P_\ell\tilde\beta_{k\ell}\right)\varphi, \varphi\rangle\,,\]
and
\[\tilde \beta_{k\ell}=-ih\int_0^{\delta/h}g^{k\ell}(y',{h}  t)f_{h,y'}\partial_kf_{h,y'}|G(y',{h}  t)|^{\frac 12}\dd t\,.\]
We notice that
\[\begin{split}
E_h^0(\varphi)&=\langle\sum_{k\ell}\left(\overline{\tilde\beta_{\ell k}}P_k+\tilde\beta_{k\ell}P_\ell\right)\varphi,\varphi\rangle-ih\langle\sum_{k\ell}\partial_\ell\tilde\beta_{k\ell}\varphi,\varphi\rangle\\
&=\langle\sum_{k\ell}\left(\overline{\tilde\beta_{\ell k}}P_k+\tilde\beta_{\ell k}P_k\right)\varphi,\varphi\rangle-ih\langle\sum_{k\ell}\partial_\ell\tilde\beta_{k\ell}\varphi,\varphi\rangle\\
&=-ih\langle\sum_{k\ell}\partial_\ell\tilde\beta_{k\ell}\varphi,\varphi\rangle\,.
\end{split}\]
\end{proof}
Let us now deal with $Q^{\mathrm{tg}}(\varphi)$. 
\begin{lemma}\label{Lemma32}
	We have
	\[Q^{\mathrm{tg}}(\varphi)=Q^{\mathrm{tg}}_0(\varphi)+R_{h} (\varphi)\,,\quad Q^{\mathrm{tg}}_0(\varphi)=\sum_{k\ell}\int_D\alpha_{k\ell}(-i{h} \partial_k-\tilde A^0_k)\varphi\overline{(-i{h} \partial_\ell-\tilde A^0_\ell)\varphi}\dd y'\,,\]
with
\[R_{{h} }(\varphi)=\sum_{k\ell}\int_D \beta_{k\ell}\left[(-i{h} \partial_k-\tilde A^0_k)\varphi\, \overline{\varphi}+\varphi \overline{(-i{h} \partial_k-\tilde A^0_k)\varphi}\right]+\gamma_{k\ell}|\varphi|^2\dd y'\,,\]
and the coefficients $\alpha_{k\ell},\beta_{k\ell},\gamma_{k\ell}$ are introduced in \eqref{eq:coef-abg}.
\end{lemma}
\begin{proof}
We have
\[\begin{split}&Q^{\mathrm{tg}}(\varphi)=\\
&\int_{\mathcal{V}_{h} }f_{{h} ,y'}^2\sum_{k\ell}g^{k\ell}(y',{h}  t)(-i{h} \partial_k-\tilde A^0_k(y'))\varphi\overline{(-i{h} \partial_\ell-\tilde A_\ell(y',{h}  t))\varphi}|g(y',{h}  t)|^{\frac 12}\dd y'\dd t\\
&+\int_{\mathcal{V}_{h} }f_{{h} ,y'}^2\sum_{k\ell}g^{k\ell}(y',{h}  t)(\tilde A_k^0-\tilde A_k)\varphi\overline{(-i{h} \partial_\ell-\tilde A_\ell(y',{h}  t))\varphi}|g(y',{h}  t)|^{\frac 12}\dd y'\dd t\,,
\end{split}\]
and then
\[\begin{split}Q^{\mathrm{tg}}(\varphi)=
\int_{\mathcal{V}_{h} }f_{{h} ,y'}^2\sum_{k\ell}g^{k\ell}(y',{h}  t)(-i{h} \partial_k-\tilde A^0_k)\varphi\overline{(-i{h} \partial_\ell-\tilde A^0_\ell)\varphi}|g(y',{h}  t)|^{\frac 12}\dd y'\dd t+R_{h} (\varphi)\,,
\end{split}\]
where
\[\begin{split}
R_{h} (\varphi)=&\int_{\mathcal{V}_{h} }f_{{h} ,y'}^2\sum_{k\ell}g^{k\ell}(y',{h}  t)(-i{h} \partial_k-\tilde A^0_k)\varphi\overline{(\tilde A^0_\ell-\tilde A_\ell)\varphi}|g(y',{h}  t)|^{\frac 12}\dd y'\dd t\\
&+\int_{\mathcal{V}_{h} }f_{{h} ,y'}^2\sum_{k\ell}g^{k\ell}(y',{h}  t)(\tilde A_k^0-\tilde A_k)\varphi\overline{(-i{h} \partial_\ell-\tilde A^0_\ell)\varphi}|g(y',{h}  t)|^{\frac 12}\dd y'\dd t\\
&+\int_{\mathcal{V}_{h} }f_{{h} ,y'}^2\sum_{k\ell}g^{k\ell}(y',{h}  t)(\tilde A_k^0-\tilde A_k)(\tilde A_\ell^0-\tilde A_\ell)|\varphi|^2|g(y',{h}  t)|^{\frac 12}\dd y'\dd t\,.
\end{split}\]
Applying the Fubini theorem, we get the result.
\end{proof}
The (self-adjoint) operator associated with $Q^{\mathrm{tg}}$, on the Hilbert space  $L^2(D)$ (with the canonical  scalar product), is
\begin{equation}\label{eq:3.21}
\mathscr{L}^{\mathrm{tg}}=\sum_{k\ell} (P_\ell \alpha_{k\ell}P_k+\beta_{k\ell} P_k+P_k\beta_{k\ell}+\gamma_{k\ell})=\sum_{k\ell} P_\ell \alpha_{k\ell}P_k+\sum_k (\hat \beta_k P_k+P_k \hat \beta_{k})+  \gamma \,,
\end{equation}
 where $\hat\beta_k=\sum\limits_{\ell}\beta_{k\ell}$ and $\gamma=\sum_{k\ell}\gamma_{k\ell}$.\\

Therefore, we arrive, modulo remainders of order $\mathcal{O}(h^3)$, at the effective operator introduced in \eqref{eq:eff-op-2D}, which can be written in the form
\begin{equation} \label{eq:3.22}
\mathscr{L}^{\mathrm{eff}}=\mathscr{L}^{\mathrm{tg}}+\mu(y',{h})-h^2\rho(y',h) \,.
\end{equation}
The min-max theorem implies that, for all $n\geq 1$,
\begin{equation}\label{eq.ub}
\mu_n(h)\leq \mu_n^{\mathrm{eff}}(h)+Ch^3 \,.
\end{equation}

\subsection{Lower bound}~\\
For every $y'$, we  introduce the  projection $\pi_{y'}$
on the ground state $f_{h,y'}$  of the transverse operator, which acts on the space $L^2( (0,\delta/h);|g(y',h t)|^{1/2}\dd t)$ as follows
\begin{equation}\label{eq:pi-y'}
\pi_{y'} f=f_{{h} ,y'}\langle f, f_{{h} ,y'}(t)\rangle_{L^2((0,\delta/h),|g(y',h t)|^{1/2}\dd t)}\,.
\end{equation}
Also we denote by $\pi_{y'}^\bot={\rm Id}-\pi_{y'}$, which is orthogonal to $\pi_{y'}$. 

Now we define the projections $\Pi$ and $\Pi^\bot$ acting on  $\psi\in L^2(\mathcal V_h)$ as follows ($\mathcal V_h$ is introduced in \eqref{eq:Vh-c})
\begin{equation}\label{defPi}
\Pi \psi (y',\cdot) =\pi_{y'} \psi (y',\cdot)f_{h,y'}(t)\varphi(y') \quad{\rm and}~\Pi^\perp \psi (y',\cdot) =\pi_{y'}^\perp  \psi (y',\cdot)\,,
\end{equation}
where we write
\[\varphi(y') =\int_0^{\delta/h} f_{{h} ,y'}(t)\overline{\psi(y',t)}\,|g(y',h t)|^{1/2}\dd t\,.\]
Note that, for all every $y'\in D$, we have 
\[\int_0^{\delta/h} \Pi\psi (y',t) \overline{ \Pi^\bot\psi(y',t)}\,|g(y',h t)|^{1/2}\dd t=0\,,\]
thereby allowing us to decompose the quadratic form $Q$ (see \eqref{defQ}) as follows
\[Q(\psi)=Q^{\mathrm{tr}}(\Pi\psi)+Q^{\mathrm{tr}}(\Pi^\perp\psi)+Q^{\mathrm{bnd}}(\Pi\psi+\Pi^\perp\psi)\,,\]
 for all $\psi\in H^1(\mathcal V_h)$ which vanishes on $y_3=\delta/h$ (see \eqref{eq:Vh-c}).
Then,
\begin{equation}\label{eq.decompose}
Q(\psi)\geq Q^{\mathrm{tr}}(\Pi\psi)-C{h} \|\Pi^\perp\psi\|^2+Q^{\mathrm{bnd}}(\Pi\psi)+Q^{\mathrm{bnd}}(\Pi^\perp\psi)+2\Re Q^{\mathrm{bnd}}(\Pi\psi,\Pi^\perp\psi)\,.
\end{equation}
We must deal with the last terms. These terms are in the form
\[\mathscr{J}_{k\ell}({h} )=\int_{\mathcal{V}_{h} }g^{k\ell}(y',{h}  t)(-i{h} \partial_k-\tilde A_k(y',{h}  t))\Pi\psi\overline{(-i{h} \partial_\ell-\tilde A_\ell(y',{h}  t))\Pi^\perp\psi}|g(y',{h}  t)|^{\frac 12}\dd y'\dd t\,.\]

\begin{lemma}
	We have
	\[\begin{split}
	|\mathscr{J}_{k\ell}({h} )|\leq&|\widetilde{\mathscr{J}^0_{k\ell}}({h} )|
	+C{h} ^2\|\Pi\psi\|\|P_\ell\Pi^\perp\psi\|+C{h} ^2\|P_k\Pi\psi\|\|\Pi^\perp\psi\|+C{h} ^2\|\Pi\psi\|\|\Pi^\perp\psi\|\\
	&+C_\varepsilon{h} ^2\|P_\ell\Pi\psi\|^2+C_\varepsilon{h} ^2\|P_k\Pi\psi\|^2+\varepsilon(\|\Pi^\perp\psi\|^2+\|P_\ell\Pi^\perp\psi\|^2)\,,
	\end{split}
	\]	
	where
	\[\widetilde{\mathscr{J}^0_{k\ell}}({h} )=\int_{\mathcal{V}_{h} }g^{k\ell}(y',0)(-i{h} \partial_k-\tilde A^0_k)\Pi\psi\overline{(-i{h} \partial_\ell-\tilde A_\ell^0)\Pi^\perp\psi}|g(y',{h}  t)|^{\frac 12}\dd y'\dd t\,,\]
	and $\mathcal V_h$ introduced in \eqref{eq:Vh-c}.
\end{lemma}
\begin{proof}
	We can proceed by following the same lines as before. 
	 Recall the projections $\pi_{y'}$, $\Pi$ and $\Pi^\bot$ introduced in \eqref{eq:pi-y'} and \eqref{defPi}, and  that $\pi_{y'}$ is an orthogonal projection with respect to the $L^2(|g|^{1/2}(y',{h}  t)\dd t)$ scalar product. 
		First, we write
	\[
	\mathscr{J}_{k\ell}({h} )=
	\widetilde{\mathscr{J}_{k\ell}}({h} )+R_{k\ell}({h} )\,,
	\]
	where
	\[\widetilde{\mathscr{J}_{k\ell}}({h} )=\int_{\mathcal{V}_{h} }g^{k\ell}(y',0)(-i{h} \partial_k-\tilde A_k(y',{h}  t))\Pi\psi\overline{(-i{h} \partial_\ell-\tilde A_\ell(y',{h}  t))\Pi^\perp\psi}|g(y',{h}  t)|^{\frac 12}\dd y'\dd t\,.\]
	Replacing $\tilde A_k$ by $\tilde A_k^0$, we get
	\[\begin{split}
	&|R_{k\ell}({h} )|\\ 
	&\leq C{h} ^2\|\Pi\psi\|\|P_\ell\Pi^\perp\psi\|+C{h} ^2\|P_k\Pi\psi\|\|\Pi^\perp\psi\|+C{h} ^3\|\Pi\psi\|\|\Pi^\perp\psi\|+C{h} \|P_k\Pi\psi\|\|P_\ell\Pi^\perp\psi\|\\
	&\leq C{h} ^2\|\Pi\psi\|\|P_\ell\Pi^\perp\psi\|+C{h} ^2\|P_k\Pi\psi\|\|\Pi^\perp\psi\|+C{h} ^3\|\Pi\psi\|\|\Pi^\perp\psi\|+C_\varepsilon{h} ^2\|P_k\Pi\psi\|^2+\varepsilon\|P_\ell\Pi^\perp\psi\|^2\,.
	\end{split}\]
	Playing the same game, we write
	\[\widetilde{\mathscr{J}_{k\ell}}({h} )=\widetilde{\mathscr{J}_{k\ell}}^0({h} )+\tilde R_{k\ell}({h} )\,,\]
with
	\[\begin{split}
	\tilde R_{k\ell}({h} )&\\
	=&\int_{\mathcal{V}_{h} }g^{k\ell}(y',0)(\tilde A_k^0-\tilde A_k(y',{h}  t))\Pi\psi\overline{(-i{h} \partial_\ell-\tilde A_\ell(y',{h}  t))\Pi^\perp\psi}|g(y',{h}  t)|^{\frac 12}\dd y'\dd t\\
	&+\int_{\mathcal{V}_{h} }g^{k\ell}(y',0)(-i{h} \partial_k-\tilde A_k^0)\Pi\psi\overline{(\tilde A_\ell^0-\tilde A_\ell(y',{h}  t))\Pi^\perp\psi}|g(y',{h}  t)|^{\frac 12}\dd y'\dd t\\
	=&\int_{\mathcal{V}_{h} }g^{k\ell}(y',0)(\tilde A_k^0-\tilde A_k(y',{h}  t))\Pi\psi\overline{(\tilde A_\ell^0-\tilde A_\ell)\Pi^\perp\psi}|g(y',{h}  t)|^{\frac 12}\dd y'\dd t\\
	&+R^1_{k\ell}({h} )+R^2_{k\ell}({h} )\,,
	\end{split}
	\]
	\[R_{k\ell}^1({h} )=\int_{\mathcal{V}_{h} }g^{k\ell}(y',0)(\tilde A_k^0-\tilde A_k(y',{h}  t))\Pi\psi\overline{(-i{h} \partial_\ell-\tilde A_\ell^0)\Pi^\perp\psi}|g(y',{h}  t)|^{\frac 12}\dd y'\dd t\,,\]
	and
	\[R_{k\ell}^2({h} )=\int_{\mathcal{V}_{h} }g^{k\ell}(y',0)(-i{h} \partial_k-\tilde A_k^0)\Pi\psi\overline{(\tilde A_\ell^0-\tilde A_\ell(y',{h}  t))\Pi^\perp\psi}|g(y',{h}  t)|^{\frac 12}\dd y'\dd t'\,.\]
	Let us estimate the remainder $\tilde R_{k\ell}({h} )$. Its first term can be estimated via the Cauchy-Schwarz inequality:
	\[\begin{split}
	\tilde R_{k\ell}({h} )&\leq C{h} ^2\|\Pi\psi\|\|\Pi^\perp\psi\|+\left|R_{k\ell}^1({h} )\right|+\left|R_{k\ell}^2({h} )\right|\\
	\end{split}\,.\]
	We have
	\[|R_{k\ell}^2({h} )|\leq C{h}  \|P_k(\Pi\psi)\|\|\Pi^\perp\psi\|\leq C_\varepsilon{h} ^2\|P_k(\Pi\psi)\|^2+ \varepsilon\|\Pi^\perp\psi\|^2\,.\]
	To estimate $R_{k\ell}^1({h} )$, we integrate by parts with respect to $y_\ell$:
	\[R_{k\ell}^1({h} )=\int_{\mathcal{V}_{h} }P_\ell(g^{k\ell}(y',0)|g(y',{h}  t)|^{\frac 12}(\tilde A_k^0-\tilde A_k(y',{h}  t)\Pi\psi)\overline{\Pi^\perp\psi}\dd y'\dd t\,.\]
	Then, 
	\[\begin{split}
	|R_{k\ell}^1({h} )|&\leq C{h}  \|P_k\Pi\psi\|\|\Pi^\perp\psi\|+C{h} ^2\|\Pi\psi\|\|\Pi^\perp\psi\|\\
	&\leq C_\varepsilon{h} ^2\|P_\ell\Pi\psi\|^2+\varepsilon\|\Pi^\perp\psi\|^2+C{h} ^2\|\Pi\psi\|\|\Pi^\perp\psi\|\,.
	\end{split}\]
\end{proof}
By computing the commutator between $\Pi$ and the tangential derivatives, and using \eqref{eq.deriv-f}, we get the following.
\begin{lemma}
	We have
	\[|\widetilde{\mathscr{J}^0_{k\ell}}({h} )|\leq C{h} ^2\left(\|\Pi\psi\|\|P_\ell\Pi^\perp\psi\|+\|P_k\Pi\psi\|\|\Pi^\perp\psi\|\right)\,.\]	
\end{lemma}
From the last two lemmas, we deduce the following.
\begin{proposition}\label{prop.mix}
 For any $\varepsilon >0$, there exist $h_\varepsilon,C_\varepsilon>0$ such that, for all $h\in (0,h_\varepsilon]$, we have 
	\[|\Re Q^{\mathrm{bnd}}(\Pi\psi,\Pi^\perp\psi)|\leq\varepsilon\left(\|\Pi^\perp\psi\|^2+\sum_{\ell}\|P_\ell\Pi^\perp\psi\|^2\right)+C_\varepsilon {h} ^2\Big(\sum_{\ell}\|P_\ell\Pi\psi\|^2+ {h} ^2\|\Pi\psi\|^2\Big)\,.\]
\end{proposition}
In the sequel, $\varepsilon$ will be selected small but {\bf fixed}, so we will drop the reference to $\varepsilon$ in the constants $C_\varepsilon $ and $h_\varepsilon$. These constants may vary from one line to another without mentioning this explicitly.

\subsection{Proof of Proposition \ref{prop.eff}}\label{sec.proofeff}
From \eqref{eq.decompose} and Proposition \ref{prop.mix}, we get, by choosing $\varepsilon$ small enough,
\[\begin{split}Q(\psi)\geq\int_D\mu(y',{h} )|\varphi(y')|^2\dd y'+(1-C{h} ^2)Q^{\mathrm{tg}}_0(\varphi)+R_{h} (\varphi)+E_h(\varphi)-C{h} ^4\|\varphi\|^2-\varepsilon \|\Pi^\perp\psi\|^2\,.\end{split}\]
Since the first eigenvalues are close to $-1$, the min-max theorem implies that
\begin{equation}\label{eq.lb0}
\mu_n({h} )\geq \tilde\mu^{\mathrm{eff}}_n({h} )-C{h} ^4\,,
\end{equation}
where $\tilde\lambda^{\mathrm{eff}}_n({h} )$ is the $n$-th eigenvalue of
\[\widetilde{\mathscr{L}}^{\mathrm{eff}}=\sum_{k\ell} ((1-C{h} ^2)P_\ell \alpha_{k\ell}P_k+\beta_{k\ell} P_k+P_k\beta_{k\ell}+\gamma_{k\ell})+\tilde \mu(y',h) \,,\]
with
\[\tilde\mu(y',h)=\mu(y',{h})-h^2\rho(y',h)\,.\]
As we can see $\widetilde{\mathscr{L}}^{\mathrm{eff}}$ is a slight perturbation of ${\mathscr{L}}^{\mathrm{eff}}$. It is rather easy to check that
\[\tilde\mu_n^{\mathrm{eff}}(h)=-1+\mathcal{O}(h)\,,\]
so that, for all normalized eigenfunction $\psi$ associated with $\tilde\mu_n^{\mathrm{eff}}(h)$, we have
\[\sum_{\ell}\|P_\ell\psi\|^2=\mathcal{O}(h)\,,\]
where we used Remark \ref{rem.abcr}. This a priori estimate, with the min-max principle, implies that
\begin{equation}\label{eq.lb1}
\tilde \mu_n^{\mathrm{eff}}(h)\geq \mu^{\mathrm{eff}}_n(h)-Ch^3\,.
\end{equation}
Proposition is a consequence of \eqref{eq.lb0}, \eqref{eq.lb1}, and \eqref{eq.ub}.

\section{Spectral analysis of the effective operator}\label{sec:efop*}
Thanks to Proposition \ref{prop.eff}, we may focus our attention on the effective operator (see \eqref{eq:eff-op-2D}) on the $L^2(D)$,
\[\mathscr{L}^{\mathrm{eff}}=\sum_{k\ell} P_\ell \alpha_{k\ell}P_k+\sum_k \hat\beta_{k} P_k+P_k\hat\beta_{k}+\gamma+\tilde\mu(y',{h} )\,,\]
where $\hat\beta_k$ and $\gamma$ were introduced in \eqref{eq:3.21}.
\subsection{A global effective operator}
In view of Remark \ref{rem.abcr}, it is natural to consider the new operator
\[\mathscr{L}^{\mathrm{eff}, 0}=\sum_{k\ell} (P_\ell \alpha^{[0]}_{k\ell}P_k+hP_\ell\alpha_{k\ell}^{[1]}P_k+h(\beta^{[1]}_{k\ell} P_k+P_k\beta^{[1]}_{k\ell})+h^2\gamma^{[2]}_{k\ell})-2\kappa(y')h+h^2\mu^{[2]}(y')\,.\]
We can prove that the rough estimates \[\mu_n^{\mathrm{eff}}(h)+1=\mathcal{O}(h)\,,\quad \mu_n^{\mathrm{eff},0}(h)=\mathcal{O}(h)\,.\]
By using the same considerations as in Section \ref{sec.proofeff}, we may check that the action of $P_\ell$ on the low lying eigenfunctions is of order $\mathcal{O}(h^{\frac 12})$, and we get the following.
\begin{proposition}\label{prop.eff-eff0}
For all $n\geq 1$, there exist $h_0>0$, $C>0$ such that, for all $h\in(0,h_0)$,
\[|\mu_n^{\mathrm{eff}}(h)-(1+\mu_n^{\mathrm{eff},0}(h))|\leq Ch^{\frac 52}\,.\]	
\end{proposition}
Therefore, we can focus on the spectral analysis of $\mathscr{L}^{\mathrm{eff}, 0}$. In order to lighten the notation, we drop the superscript $[j]$ in the expression of $\mathscr{L}^{\mathrm{eff}, 0}$ when it is not ambiguous. Thus,
\[\mathscr{L}^{\mathrm{eff}, 0}=\sum_{k\ell} \left(P_\ell \alpha_{k\ell}P_k+hP_\ell\alpha^{[1]}_{k\ell}P_k\right)+h\sum_{k=1}^2(\tilde\beta_{k} P_k+P_k\tilde\beta_{k})-2\kappa(y')h+h^2V(y')\,,\]
where $\hat\beta_k$, $\gamma$ are introduced in \eqref{eq:3.21} and 
\[V(y')=\mu(y')+\gamma(y')\,.\]
We recall that this operator is equipped with the Dirichlet boundary conditions on $\partial D$. In fact, by using a partition of the unity, as in Section \ref{sec.2}, we can prove that
\[\mu_n^{\mathrm{eff}}(h)=h(\min_{y'\in D}\sqrt{\det \alpha}\,|\curl \tilde A^0|-2\kappa(y'))+o(h)\,.\]
Note that 
\[\sqrt{\det \alpha}\,\curl \tilde A^0=\mathbf{B}\cdot \mathbf{n}\,.\]
Thus,
\[\mu_n^{\mathrm{eff}}(h)=h(\min(|\mathbf{B}\cdot\mathbf{n}|-2\kappa)+o(h)\,.\]
Due to our assumption that the minimum of $|\mathbf{B}\cdot\mathbf{n}|-2\kappa$ is unique, we deduce, again as in Section \ref{sec.2}, that the eigenfunctions are localized, in the Agmon sense, near $y'=0$ (the coordinate of $x_0$ on the boundary).

This invites us to define a global operator, acting on $L^2(\mathbb{R}^2)$. Consider a ball $D_0\subset D$ centered at $y'=0$. Outside $D_0$, we can smoothly extend the (informly in $y'$) positive definite matrix $\alpha$ to $\mathbb{R}^2$ so that the extension is still definite positive (uniformly in $y'$) and constant outside $D$.
Then, consider the function
\[b(y')=\sqrt{\det \alpha}\,\tilde b(y')\,,\qquad \tilde b:=\curl \tilde A^0\,.\]
Its extension may be chosen so that the extended function has still a unique and non-degenerate minimum (not attained at infinity) and is constant outside $D$. With these two extensions, we have a natural extension of $\tilde b$ to $\mathbb{R}^2$. We would like to extend $\tilde A^0$, but it is not necessary. We may consider an associated smooth vector potential $\hat A^0$ defined on $\mathbb{R}^2$ and growing at most polynomially (as well as all its derivatives). Up to change of gauge on $D$ and thanks to the rough localization near $y'=0$, the low-lying eigenvalues of $\mathscr{L}^{\mathrm{eff}, 0}$ coincide modulo $\mathcal{O}(h^\infty)$ with the one of $\widehat{\mathscr{L}}^{\mathrm{eff}, 0}$ defined by replacing $\tilde A^0$ by $\hat A^0$.

In the same way, we extend $\kappa$, $V$ and $\beta$.

Modulo $\mathcal{O}(h^\infty)$, we may consider
\[\mathscr{L}^{\mathrm{eff}, 0}=\sum_{k\ell} \left(P_\ell \alpha_{k\ell}P_k+hP_\ell\alpha^{[1]}_{k\ell}P_k\right)+h\sum_{k=1}^2(\tilde\beta_{k} P_k+P_k\tilde\beta_{k})-2\kappa(y')h+h^2V(y')\,,\]
acting on $L^2(\mathbb{R}^2)$, where $\alpha$, $\beta$, $\kappa$, $V$ are the extended functions, and where $P_\ell=-ih\partial_\ell-\tilde A^0$.
\subsection{Semiclassical analysis: proof of Theorem \ref{thm:ev-c}}~\\
Having the effective operator in hand, we determine in Theorem~\ref{thm:splitting-eff} below the asymptotics for the low-lying eigenvalues. In turn this yields Theorem~\ref{thm:ev-c} after collecting \eqref{eq:ev-g-h}, \eqref{eq.ub}, \eqref{eq.lb1} and  Proposition~\ref{prop.eff-eff0}. 

Note that  the situation considered in \cite{HM3d} and \cite{HKo} is different. In our situation, we determine an effective \emph{two dimensional} global operator (see Proposition~\ref{prop.eff-eff0}), and we get the spectral asymptotics from those of the effective operator. Our effective operator inherits a natural magnetic field as well, whose analysis goes in the same spirit as for the pure  magnetic Laplacian (see \cite{HKo1,RVN15}). 
  
We have
\[\mathscr{L}^{\mathrm{eff},0}=\mathrm{Op}^{\mathrm{W}}_{h} \left(H^{\mathrm{eff}} \right)\,,\]
where
\begin{multline*}
H^{\mathrm{eff}} =\sum_{k\ell}\alpha_{k\ell}(p_\ell-\tilde A^0_\ell)(p_k-\tilde A_k^0)+h\sum_{k\ell}\alpha_{k\ell}^{[1]}(p_\ell-\tilde A^0_\ell)(p_k-\tilde A_k^0)\\ +2h\sum_{k=1}^2\tilde\beta_{k}(p_k-\tilde A^0_k)-2\kappa(y')h\
 +h^2\tilde V(y') )\,,
\end{multline*} 
for some new $\tilde V$.

The principal symbol of $\mathscr{L}^{\mathrm{eff},0}$ is thus 
\[H(q,p)=\sum_{k\ell}\alpha_{k\ell}(p_\ell-A_\ell(q))(p_k-A_k(q))=:\|p-A(q)\|^2_\alpha\,,\]
where we dropped the tildas and the superscript $0$ to lighten the notation.

Theorem \ref{thm:ev-c} is a consequence of the following theorem (and of \eqref{eq.locbnd} and Propositions \ref{prop.eff} and \ref{prop.eff-eff0}), recalling \eqref{defQ} and \eqref{eq:ev-g-h} (with $\sigma=1$).
\begin{theorem}\label{thm:splitting-eff}
Let $n\geq 1$.	There exists $c_1\in\R$ such that
\[\mu_n^{\mathrm{eff},0}(h)=h\min_{x\in \partial \Omega}(|\mathbf{B}\cdot\mathbf{n}(x) |-2\kappa(x) )+h^2(c_0(2n-1)+c_1)+\mathcal{O}(h^3)\,,\]
with
\[c_0=\frac{\sqrt{\det(\mathrm{Hess}_{x_0}(|\mathbf{B}\cdot\mathbf{n}|-2\kappa ))}}{2|\mathbf{B}\cdot\mathbf{n}(x_0)|} \,.\]
\end{theorem}
\begin{proof}
The proof closely follows the same lines as in \cite{RVN15}. Let us only recall the strategy without entering into detail.	
	
Let us consider the characteristic manifold
\[\Sigma=\{(q,p)\in \R^4 : H(q,p)=0)\}=\{(q,p)\in\R^4 : p=\Ab(q)\}\,.\]
Considering the canonical symplectic form $\omega_0=\dd p\wedge \dd q$, an easy computation gives
\[(\omega_0)_{|\Sigma}= B \, \dd q_1\wedge q_2\,,\quad B(q)=\partial_1A_2-\partial_2 A_1\,.\]
Our assumptions imply that $B\geq B_0>0$. This suggests to introduce the new coordinates
\[ q=\varphi^{-1}(\tilde q)\,,\quad \mbox{ with }q_1=\tilde q_1\,,\quad q_2=\int_0^{\tilde q_2}B(\tilde q_1,u)\dd u\,.\]
We get
\[
\varphi^*(\omega_0)_{|\Sigma}=\dd \tilde q_1\wedge \dd \tilde q_2
\,.\]
This allows to construct a quasi symplectomorphism which sends $\Sigma$ onto $\{x_1=\xi_1=0\}$. 
Indeed, consider 
\[\Psi : (x_1, x_2, \xi_1,\xi_2)\mapsto j(x_2,\xi_2)+x_1\mathbf{e}(x_2,\xi_2)+\xi_1\mathbf{f}(x_2,\xi_2)\,,\]
with
\[j(x_2,\xi_2)=(\varphi(x_2,\xi_2), A(\varphi(x_2,\xi_2)))\in\Sigma\,,\]
and
\[\mathbf{e}(x_2,\xi_2)=B^{-\frac 12}(\mathbf{e}_1,\dd \Ab^T(\mathbf{e}_1))\,,\quad \mathbf{f}(x_2,\xi_2)=B^{-\frac 12}(\mathbf{e}_2,\dd \Ab^T(\mathbf{e}_2))\,,\]
where $(\dd\Ab)^{\mathrm{T}}$ is the usual transpose of the Jacobian matrix $\dd\Ab$ of $\Ab$.

On $x_1=\xi_1=0$, we have $\Psi^*\omega_0=\omega_0$. The map $\Psi$ can be slightly modified (by composition with a map tangent to the identity) so that it becomes symplectic.

Let us now describe $H$ in the coordinates $(x,\xi)$,\emph{i.e}, the new Hamiltonian $H\circ\Psi$. To do that, it is convenient to estimate $\dd^2 H$ on $T\Sigma^{\perp_{\omega_0}}$. We have
\[\dd^2 H((P,\dd \Ab^T(P)),(P,\dd \Ab^T(P)))=2B^2\|P\|^2_{\alpha}\,.\]
Then, by Taylor expansion near $\{x_1=\xi_1=0\}$,
\[H\circ\Psi(x,\xi)=H(j(z_2)+x_1\mathbf{e}+\xi_1\mathbf{f})=B(\varphi(x_2,\xi_2))\|x_1\mathbf{e}_1+\xi_1\mathbf{e}_2\|_\alpha^2+\mathcal{O}(|z_1|^3)\,.\]
Clearly, $(x_1,\xi_1)\mapsto B(\varphi(x_2,\xi_2))\|x_1\mathbf{e}_1+\xi_1\mathbf{e}_2\|_\alpha^2$ is a quadratic form with coefficients depending on $z_2$. For $z_2$ fixed, this quadratic form can be transformed by symplectomorphism into 
$B^2(\varphi(x_2,\xi_2))\sqrt{\det \alpha}|z_1|^2$. By perturbing this symplectomorphism, we find that there exists a symplectomorphism $\tilde\Psi$ such that
\[H\circ\tilde\Psi(x,\xi)=B(\varphi(x_2,\xi_2))\sqrt{\det \alpha}|z_1|^2+\mathcal{O}(|z_1|^3)\,.\]
By using the improved Egorov theorem, we may find a Fourier Integral Operator $U_h$, microlocally unitary near $\Sigma$, such that
\[U_h^*\mathscr{L}^{\mathrm{eff}} U_h=\mathrm{Op}^{\mathrm{W}}_h\widehat{ H}^{\mathrm{eff}}\,,\]
with
\[\widehat{H}^{\mathrm{eff}}=B(\varphi(x_2,\xi_2))\sqrt{\det \alpha}|z_1|^2-2h\kappa(\varphi(x_2,\xi_2))+\mathcal{O}(|z_1|^3+h|z_1|+h^2)\,,\]
locally uniformly with respecto to $(x_2,\xi_2)$.
This allows to implement a Birkhoff normal form,  as in \cite[Sections 2.3 \& 2.4]{RVN15}, and we get another Fourier Integral Operator $V_h$ such that 
\begin{equation}\label{eq.FIOV}
V_h^*\mathrm{Op}^{\mathrm{W}}_h\widehat{ H}^{\mathrm{eff}} V_h=\mathrm{Op}^{\mathrm{W}}_h\left(\check{ H}^{\mathrm{eff}}(\mathscr{I}_h, z_2, h)\right)+\mathrm{Op}^{\mathrm{W}}_h r_h\,,
\end{equation}
where $\mathscr{I}_h=\mathrm{Op}^{\mathrm{W}}_h (|z_1|^2)$ and $r_h=\mathcal{O}(|z_1|^\infty+h^\infty)$ (uniformly with respect to $z_2$). The first pseudo-differential in the R. H. S. of \eqref{eq.FIOV} is the quantization with respect to $(x_2,\xi_2)$ of the (operator) symbol $\check{ H}^{\mathrm{eff}}(\mathscr{I}_h, z_2, h)$ (commuting with the harmonic oscillator $\mathscr{I}_h$).
Moreover, $\check{H}^{\mathrm{eff}}$ satisfies
\[\check{ H}^{\mathrm{eff}}(I, z_2, h)=IB(\varphi(x_2,\xi_2))\sqrt{\det \alpha}-2h\kappa (\varphi(x_2,\xi_2))+\mathcal{O}(I^2+hI+h^2)\,.\]
We can prove that the eigenfunctions of $\mathscr{L}^{\mathrm{eff},0}$ (corresponding to the low lying spectrum) are microlocalized near $\Sigma$ and localized near the minimum of $B-2\kappa $, and also that the one of $\mathrm{Op}^{\mathrm{W}}_h\check{ H}^{\mathrm{eff}}$ are microlocalized near $0\in\R^4$. More precisely, for some smooth cutoff function on $\R$, $\chi$ and equaling $1$ near 0, and if $\psi$ is a normalized eigenfunction associated with an eigenvalue of order $h$, we have
\[\mathrm{Op}^{\mathrm{W}}_h \chi(h^{-2\delta}|z_1|^2)\psi=\psi+\mathscr{O}(h^\infty)\,,\quad \mathrm{Op}^{\mathrm{W}}_h \chi(|z_2|^2)\psi=\psi+\mathscr{O}(h^\infty)\,,\quad \delta\in\left(0,\frac 12\right)\,.\]
This implies that the low-lying eigenvalues of $\mathscr{L}^{\mathrm{eff},0}$ coincide modulo $\mathcal{O}(h^\infty)$ with the one of $\mathrm{Op}^{\mathrm{W}}_h\left(\check{ H}^{\mathrm{eff}}(\mathscr{I}_h, z_2, h)\right)$. By using the Hilbert basis of the Hermite functions, the low-lying eigenvalues are the one of $\mathrm{Op}^{\mathrm{W}}_h\check{ H}^{\mathrm{eff}}(h,z_2,h)$. Note that
\[\check{ H}^{\mathrm{eff}}(h,z_2,h)=h\left[B(\varphi(x_2,\xi_2))\sqrt{\det \alpha}-2\kappa (\varphi(x_2,\xi_2))\right]+\mathcal{O}(h^2)\,.\]
The non-degeneracy of the minimum of the principal symbol and the harmonic approximation give the conclusion.
\end{proof}
\appendix
\section{The constant curvature case}\label{sec:ball}

We treat here the case of the unit ball, $\Omega=\{x\in\R^3~:~|x|<1\}$, when the magnetic field is uniform and given by
\begin{equation}\label{eq:B=cst}
\Bb=(0,0,b)\quad{\rm with}~b>0\,.
\end{equation} 

\subsection{The critical regime}

 In the critical regime, where $\sigma=1$, the asymptotics in \eqref{eq:ev-critical} becomes  (see Remark \ref{rem:rem-c-r})
\begin{equation}\label{eq:ev-sec5}
\mu(h,\Bb)= -1-2h+\mathcal O(h^{6/5})\,,
\end{equation}
 but the magnetic field contribution is kept in the remainder term.

The contribution of the magnetic field is actually related to the ground state energy of the Montgomery  model \cite{Mont}
\begin{equation}\label{eq:Mont-op}
\lambda(\zeta)=\inf_{u\not=0}\int_\R\Big(|u'(s)|^2+\Big(\zeta+\frac{s^2}{2}\Big)^2|u(s)|^2 \Big)\,\dd s\quad(\zeta\in\R)\,.
 \end{equation}
 There exists a \emph{unique} $\zeta_0<0$ such that \cite{H-mont}
 \begin{equation}\label{eq:Mont-op-min}
 \nu_0:=\inf_{\zeta\in\R}\lambda(\zeta)=\lambda(\zeta_0)>0\,.
 \end{equation}
\begin{thm}\label{thm:ball-c}
\[\mu(h,\Bb)=-1-2h+ b^{4/3}h^{4/3}\nu_0+o(h^{4/3})\,.\]
\end{thm}
Our approach to derive an effective Hamiltonian as in  Theorem~\ref{thm:ev-c}  do not apply in the ball case. As in \cite{HM3d},  the ground states do concentrate  near the circle
\[S=\{x=(x_1,x_2,0)\in\R^3~:~x_1^2+x_2^2=1\}\,.\] 
However, the ground states  do not concentrate near a single point of $S$, since the curvature is constant. The situation here is closer to that of the Neumann problem for the 3d ball \cite{FP-ball}. 

We can improve the localization of the ground states  near the set $S$, thanks  to the energy lower bound in \eqref{eq:qh-en-lb} and the asymptotics in \eqref{eq:ev-sec5}. In fact, any $L^2$-normalized ground state $u_h$ decays away from the set $S$ as follows.

\begin{proposition}\label{prop:dec-ball}
There exists positive constants $C,h_0$ such that, for all $h\in(0,h_0)$,
\[
\int_\Omega\left(|u_h|^2+ |(h\nabla-i\Ab) u_h|^2\right)\exp\left( \frac{{\rm dist}(x,S)}{h^{1/5}}\right)\dd x\leq C\,. 
\]
\end{proposition}
\begin{proof}
Consider the function $\Phi(x)=\exp\left( \frac{{\rm dist}(x,S)}{h^{1/5}}\right)$. It satisfies
\[h^2|\nabla\Phi|^2=h^{8/5} |\Phi|^2~{\rm a.e.}\]
We write
\begin{equation}\label{eq:app-IMS}
q_h(\Phi u_h)-h^2\int_\Omega |\nabla \Phi |^2|u_h|^2\,dx=\underset{< 0}{\underbrace{\mu(h)}}\| \Phi   u_h\|_{L^2(\Omega)}^2\,,
\end{equation}
then we use \eqref{eq:ev-sec5} and \eqref{eq:qh-en-lb}. We get
\begin{multline*}
h\int_{\{{\rm dist}(x,
\partial\Omega)<h^{2/5}\} }\Big(|\Bb\cdot \nf(p(x))|-1-\tilde C h^{1/5}-h^{8/5}\Big)|\Phi u_h|^2\dd x
\\ \leq 
\int_{\{{\rm dist}(x,
\partial\Omega)\geq h^{2/5}\} } \big( |(h\nabla-i\Ab)\Phi u_h|^2+ h^{8/5}|\Phi u_h|^2\big)\dd x\,.\end{multline*}
Now we use the decay away from the boundary, \eqref{eq:dec-ef}, to estimate the term on the right hand side of the above inequality. We obtain, for  $h$ sufficiently small,
\[h\int_{\{{\rm dist}(x,
\partial\Omega)<h^{2/5}\} }\Big(|\Bb\cdot \nf(p(x))|-1-\tilde C h^{1/5}-h^{8/5}\Big)|\Phi u_h|^2\dd x\leq 
\exp(-h^{-1/5})\,.\]
The function  $\Bb\cdot \nf -1$ vanishes linearly on $S$,   so  $|\Bb\cdot\nf(p(x))|\geq 1+c\,{\rm dist}(x,\partial S)$ for a positive constant $c$. This yields
\[h\int_{\{{\rm dist}(x,
\partial\Omega)<h^{2/5}\} }\Big(c\,{\rm dist}(x,\partial S)-\tilde C h^{1/5}-h^{3/5}\Big)|\Phi u_h|^2\dd x\leq 
\exp(-h^{-1/5})\,.\]
Since $|\Phi|\leq \exp(3c^{-1}\tilde C)$ for ${\rm dist}(x,S)\leq 3c^{-1}\tilde C h^{1/5}$, the foregoing estimate yields
\begin{multline*}\int_{\{{\rm dist}(x,
\partial\Omega)<h^{2/5}\} }|\Phi u_h|^2\dd x \\ 
\leq 
\tilde C^{-1} h^{-6/5}\left( \exp(-h^{-1/5})+ 2\tilde C  h^{6/5}\int_{\{{\rm dist}(x,
S)<3c^{-1}\tilde Ch^{1/5}\} }|\Phi u_h|^2\dd x\right)=\mathcal O (1)\,.
\end{multline*}
Thanks to \eqref{eq:dec-ef}, we get
\[\| \Phi u_h\|^2_{L^2(\Omega)}=\mathcal O (1)\,.\]
Implementing this into \eqref{eq:app-IMS} finishes the proof.
\end{proof}

\begin{rem}\label{rem:dec-ball}
As a consequence of Proposition~\ref{prop:dec-ball} and the decay estimate in \eqref{eq:dec-ef},   we deduce that, for any $n\in \mathbb N$, there exist positive constants $ C_n,h_n>0$ such that, for all $h\in(0,h_n)$,
\begin{equation}\label{eq:dec-s}
\int_\Omega \big({\rm dist}(x,S)\big)^n\left(|u_h|^2+ |(h\nabla-i\Ab) u_h|^2\right)\dd x\leq C_n h^{n/5}\,,
\end{equation}
and
\begin{equation}\label{eq:dec-t}
\int_\Omega \big({\rm dist}(x,\partial\Omega)\big)^n\left(|u_h|^2+ |(h\nabla-i\Ab) u_h|^2\right)\dd x\leq  C_n h^{n}\,. 
\end{equation}
\end{rem}

In spherical coordinates, 
\[\R_+\times[0,2\pi)\times (0,\pi)\ni (r,\varphi,\theta)\mapsto x=(r\cos\varphi\sin\theta,r\sin\varphi\sin\theta,r\cos\theta)\,,\]
 the quadratic form and $L^2$-norm are
\begin{multline*}
\mathfrak q_h(u)
=\\
\int_0^{2\pi}\int_0^\pi\int_0^{1}
\left( |h\partial_r\tilde u|^2+\frac1{r^2}|h\partial_\theta\tilde u|^2+\frac1{r^2\sin^2\theta}\Bigl|\Bigl(h\partial_\varphi-i\frac{br^2}{2}\sin\theta\Bigr)\tilde u\Bigr|^2 \right)r^2\sin\theta \,  \dd r\dd\theta \dd\varphi\\
-h
\int_0^{2\pi}\int_0^\pi |\tilde u|^2_{/_{r=1}} \sin\theta\, \dd\theta \dd\varphi \,,
\end{multline*}
\[
\|u\|_{L^2(\Omega)}^2=\int_0^{2\pi}\int_0^\pi\int_0^{1} |\tilde u|^2r^2\sin\theta\, \dd r\dd\theta \dd\varphi\,,
\]
where 
\[\tilde u(r,\varphi,\theta)=u(x)\,.\]
Note that the distances to the boundary and to the set $S$ are expressed as follows
\[{\rm dist}(x,\partial\Omega)=1-r~ {\rm and}\quad {\rm dist}(x,S)=\cos\theta\,.\]
Let $\rho\in(1/5,1)$ and consider $\hat S_{\rho}=\{(r,\varphi,\theta)~:~1-h^\rho< r<1,~0\leq\varphi<2\pi~\&~|\theta-\frac\pi2|<h^\rho\}$. 
We introduce the function
\begin{equation}\label{eq:v-ball}
v(r,\varphi,\theta)=\chi\Big(h^{-\rho}(1-r)\Big)\chi\Big(h^{-\rho}\big(\theta-\frac\pi2\big)\Big) \tilde u_h(r,\varphi,\theta)\,,
\end{equation}
with $\chi\in C_c^\infty(\R;[0,1])$, ${\rm supp}\chi\subset(-1,1)$ and $\chi=1$ on $[-\frac12,\frac12]$.

Then, 
by the exponential decay of the ground state $u_h$,
\begin{multline*}
\mu(h,\Bb)=q_h(u_h)=\\
\int_{\hat S_\rho}
\left( |h\partial_r v|^2+\frac1{r^2}|h\partial_\theta v|^2+\frac1{r^2\sin^2\theta}\Bigl|\Bigl(h\partial_\varphi-i\frac{br^2}{2}\sin\theta\Bigr)v\Bigr|^2 \right)r^2\sin\theta\, \dd r\dd\theta \dd\varphi\\
-h
\int_{\hat S_{\rho}\cap\{r=1\}} |v|^2 \sin\theta \, \dd\theta \dd\varphi  +\mathcal O(h^\infty)\,.
\end{multline*}
In $\hat S_\rho$, it holds
\[r=1-{\rm dist}(x,\partial\Omega)=\mathcal O(h^\rho)~{\rm and}~\sin\theta=\cos\left(\theta-\frac\pi2\right)=1-\frac{1}2{\rm dist}(x,S)^2+\mathcal O\Big({\rm dist}(x,S)^4\Big)\,.\]
We choose $\rho=\frac{13}{60}\in(\frac15,\frac16)$. It results then from \eqref{eq:dec-s} and \eqref{eq:dec-t}, 
\begin{multline*}
\mu(h,\Bb)\geq 
\int_{\hat S_\rho}
|h\partial_r v|^2r^2\sin\theta \, \dd r\dd\theta \dd\varphi-h
\int_{\hat S_{\rho}\cap\{r=1\}} |v|^2 \sin\theta \, \dd\theta \dd\varphi \\
+(1-h^{\frac1{30} })\int_{\hat S_\rho}\left(|h\partial_\theta v|^2+\Big|\Big(h\partial_\varphi -i\frac{b}{2}\Big(1-\frac{1}2\Big(\theta-\frac\pi2\Big)^2\Big)v\Big|^2 \right)r^2\, \dd r\dd\theta \dd\varphi
 +\mathcal O(h^{\frac{8}5-\frac1{30}})\,.
\end{multline*}
Using \eqref{eq:1D-eff-n} with $\sigma=1$ and $\kappa\equiv1$, we get
\[\int_{\hat S_\rho}
|h\partial_r v|^2r^2\sin\theta \, \dd r\dd\theta \dd\varphi-h
\int_{\hat S_{\rho}\cap\{r=1\}} |v|^2 \sin\theta \, \dd\theta \dd\varphi \geq -1-2h+\mathcal O(h^2)\,.\]
It remains to study the quadratic form
\[q^{\rm tg}(v)=\int_{\hat S_\rho}\left(|h\partial_\theta v|^2+\Big|\Big(h\partial_\varphi -i\frac{b}{2}\Big(1-\frac{1}2\Big(\theta-\frac\pi2\Big)^2\Big)v\Big|^2 \right)r^2\,\dd r\dd\theta \dd\varphi\,.\]
Decomposing $v$ in Fourier modes, $v=\sum\limits_{m\in\mathbb Z}v_m e^{im\varphi}$, and using the change of variable
\[s=\left(\frac{b}2\right)^{1/3}h^{-1/3} \Big(\theta-\frac\pi2\Big)\,,\]
we obtain 
\[
q^{\rm tg}(v)=
h^{4/3}\left(\frac{b}{2}\right)^{2/3}\sum_{m\in\mathbb Z}\int_{1-h^\rho}^1r^2dr\int_\R
 \left(|\partial_s v_m|^2+\Big|\Big(\zeta_{m,h}-\frac12s^2\Big)^2v_m\Big|^2 \right)\,\left(\frac{b}{2}\right)^{-1/3}h^{1/3}ds\,.
\]
where
\[\zeta_{m,h}= \frac{2mh}{b}-1\]
We can now bound from below the foregoing quadratic form by the ground state energy $\nu_0$ of the Montgomery model. We end up with
\[q^{\rm tg}(v)\geq \nu_0 \left(\frac{b}{2}\right)^{2/3} h^{4/3}\int_{1-h^\rho}^1r^2dr\int_0^{2\pi}\int_\R|v|^2\dd\theta \dd\varphi=
  \nu_0 \left(\frac{b}{2}\right)^{2/3} h^{4/3}+\mathcal O(h^{\frac43+\rho})\,.\]
A matching upper bound can be obtained by constructing a suitable trial state related to the Montgomery model:
\[ v = \chi\Big(h^{-\rho}(1-r)\Big)\chi\Big(h^{-\rho}\big(\theta-\frac\pi2\big)\Big) w\]
where $\chi$ is as in \eqref{eq:v-ball}, $\rho=\frac{13}{60}$ and
\[w= \exp\left( i \frac{(b\zeta_0+1)\varphi }{h}  \right) u_0\big(h^{-1/2}(r-1) \big)f_{\zeta_0}\left( \left(\frac{b}2\right)^{1/3}h^{-1/3} \Big(\theta-\frac\pi2\Big)\right)\,. \]
Here $u_0(\tau)= \sqrt{2}\exp(-\tau )$ and $f_{\zeta_0}$ is the positive ground state of the Montgomery model in \eqref{eq:Mont-op}  for $\zeta=\zeta_0$ introduced in  \eqref{eq:Mont-op-min}.
\subsection{$h$-Bounded fields}

We consider now the regime where $\sigma=0$ in \eqref{eq:b} and $\Bb$ is given as in \eqref{eq:B=cst}. The relevant semiclassical parameter is then $h=\gamma^{-\frac 12}$ and the eigenvalue $\lambda(\gamma,\bb)$ is given as follows
\[\lambda(\gamma,\bb)=h^{-2} \mu(h,\Bb)\]
where 
$ \mu(h,\Bb)$ is now the ground state energy of the quadratic form
\begin{equation}\label{eq:qf-b}
\mathfrak q_h^b(u)=\int_\Omega|(-ih\nabla+bh\Ab_0)u(x)|^2\dd x-h^{3/2}\int_{\partial\Omega}|u(x)|^2\dd s(x)\,.
\end{equation}
The ground state energy $\mu(h,\Bb)$  depends on the magnetic field through the following effective eigenvalue,
\[
\lambda_m(b)=\inf_{f\in \mathcal D_m\setminus\{0\}}\frac{q_{m,b}(f)}{\|f\|_{\mathcal H}^2}\,,
\]
where $\mathcal H=L^2\big((0,\pi);\sin\theta\, \dd\theta\big)$,
\[\mathcal D_m=\begin{cases}\{f\in \mathcal H~:~\frac1{\sin\theta}\,f,f'\in\mathcal H\}&{\rm if}~m\not=0\\
\{f\in \mathcal H~:~f'\in\mathcal H\}&{\rm if~}m=0\end{cases}\]
and
\[q_{m,b}(f)=
\int_0^\pi\left(|f'(\theta)|^2+\left(\frac{m}{\sin\theta}-\frac{b}2\right)^2|f|^2\right)\,\sin\theta \, \dd\theta\,.
\]

\begin{thm}\label{thm:c-ev} 
The eigenvalue $\mu(h,\Bb)$ satisfies as $h\to0_+$,
\[\mu(h,\Bb)=-h+2h^{3/2} +h^2 \mathfrak e(b)+o(h^2)\,,\]
where 
\begin{equation}\label{eq:en-eff}
\mathfrak e(b)=\inf_{m\in\mathbb Z}\lambda_m(b)\,.
\end{equation}
\end{thm}

The effective eigenvalue, $\lambda_m(b)$ for $m=0$, satisfies $\lambda_0(b)=\frac{b^2}4$ with the  corresponding  ground state $f_{0,b}\equiv 1$.

The ground states decay exponentially away from the boundary, so we may write
\begin{equation}\label{eq:tmu}
\mu(h,b)=\tilde\mu(h,b,\rho)+\mathcal O(h^\infty)
\end{equation}
where $\rho\in(0,\frac12)$ is fixed and $\tilde\mu(h,b,\rho)$ is the eigenvalue on the spherical shell
\[\Omega_h=\{ 1-h^\rho<r<1\}\]
with Dirichlet condition on the interior boundary $\{r=1-h^\rho\}$ and  defined via the following quadratic form (expressed in spherical coordinates)
\begin{align}\label{eq:qf-spherical}
\tilde{\mathfrak q}_h^{b,\rho}(u)=&
h^2\int_0^{2\pi}\int_0^\pi\int_{1-h^\rho}^{1}
\left( |\partial_ru|^2+\frac1{r^2}|\partial_\theta u|^2+\frac1{r^2\sin^2\theta}\Bigl|\Bigl(\partial_\varphi-i\frac{br^2}{2}\sin\theta\Bigr)u\Bigr|^2 \right)r^2\,\sin\theta\,\dd r\dd\theta \dd\varphi \nonumber\\
&-h^{3/2}
\int_0^{2\pi}\int_0^\pi |u|^2_{/{r=1}} \sin\theta \,  \dd\theta \dd\varphi \,.
\end{align}
We decompose into Fourier modes (with respect to $\varphi\in[0,2\pi)$), and get the family of quadratic forms indexed by $m\in\mathbb Z$,
\begin{multline}\label{eq:qf-F-m}
\tilde{\mathfrak q}_{h,m}^{b,\rho}(u_m)=
h^2\int_0^\pi\int_{1-h^\rho}^{1}
\left( |\partial_ru_m|^2+\frac1{r^2}|\partial_\theta u_m|^2+\frac1{r^2}\Bigl|\Bigl(\frac{m}{\sin\theta}-\frac{br^2}{2}\Bigr)u_m\Bigr|^2 \right)r^2\,\sin\theta\,\dd r\dd\theta \\
-h^{3/2}
\int_0^\pi |u_m|^2_{/_{r=1}} \sin\theta \, \dd\theta \,.
\end{multline}
Finally, we introduce the large parameter
\begin{equation}\label{eq:delta}
\delta=h^{\rho-\frac12}
\end{equation}
and the change of variable, $r\mapsto t= h^{-1/2}(1-r)$, to obtain 
the new quadratic form
\begin{multline}\label{eq:qf-F-m*}
\hat{\mathfrak q}_{h,m}^{b,\rho}(v)=
\int_0^\pi\int_{0}^{\delta}
\left( |\partial_t v|^2+\frac{h}{(1-h^{1/2}t)^2}|\partial_\theta v|^2\right.
\\\left.+\frac{h}{(1-h^{1/2}t)^2}\Bigl|\Bigl(\frac{m}{\sin\theta}-\frac{b(1-h^{1/2}t)^2}{2}\Bigr)v\Bigr|^2 \right)(1-h^{1/2}t)^2\sin\theta \, \dd t\dd\theta -
\int_0^\pi |v|^2_{/_{t=0}}  \sin\theta \, \dd\theta \,.
\end{multline}
Using \cite[Sec.~2.6]{KS}, 
 we  write a lower bound for the  transversal quadratic form as follows
\[
\int_{0}^{\delta}
 |\partial_t v|^2 (1-h^{1/2}t)^2 \dd t -
 |v|^2\Big|_{t=0} \geq \big(-1-2h^{1/2} -h+ o(h)\big)\int_{0}^{\delta}
 | v|^2 (1-h^{1/2}t)^2 \dd t \,.\]
As for the tangential quadratic form, we bound it from below using the effective eigenvalue $\lambda_m(b)$ as follows
\begin{multline*}
\int_0^\pi \left(\frac{h}{(1-h^{1/2}t)^2}|\partial_\theta v|^2+\frac{h}{(1-h^{1/2}t)^2}\Bigl|\Bigl(\frac{m}{\sin\theta}-\frac{b(1-h^{1/2}t)^2}{2}\Bigr)v\Bigr|^2 \right)\sin\theta \, \dd\theta\\ \geq \big( h +o(h)\big)\lambda_m(b)\int_0^\pi | v|^2\sin\theta \, \dd\theta\,.
\end{multline*}
Inserting the two foregoing lower bounds into \eqref{eq:qf-F-m*}, minimizing over $m\in\mathbb Z$, we get the lower bound part in Theorem~\ref{thm:c-ev}.

As for the upper bound part in Theorem~\ref{thm:c-ev}, we use the trial state $v$ defined in the spherical coordinates as follows (see   \cite[Sec.~2.6]{KS})
\[\tilde v(r,\theta,\varphi)=\sqrt{2} \left(1+\Big(\frac{(1-r)^2 }{8h}-\frac14 \Big) \right)e^{h^{-1/2}(r-1) }\chi\big(h^{-\rho}(1-r)\big) f(\theta)\,,\] 
where $\chi$ is a cut-off function.  The function $f\in\mathcal D_m\setminus\{0\}$ is arbitrary. We compute
$q_h^b(v)$ introduced in \eqref{eq:qf-b}. We first minimize over $f$,  then over $m$, and get the desired  upper bound.\\

\begin{center}
\begin{figure}
\includegraphics[scale=0.5]{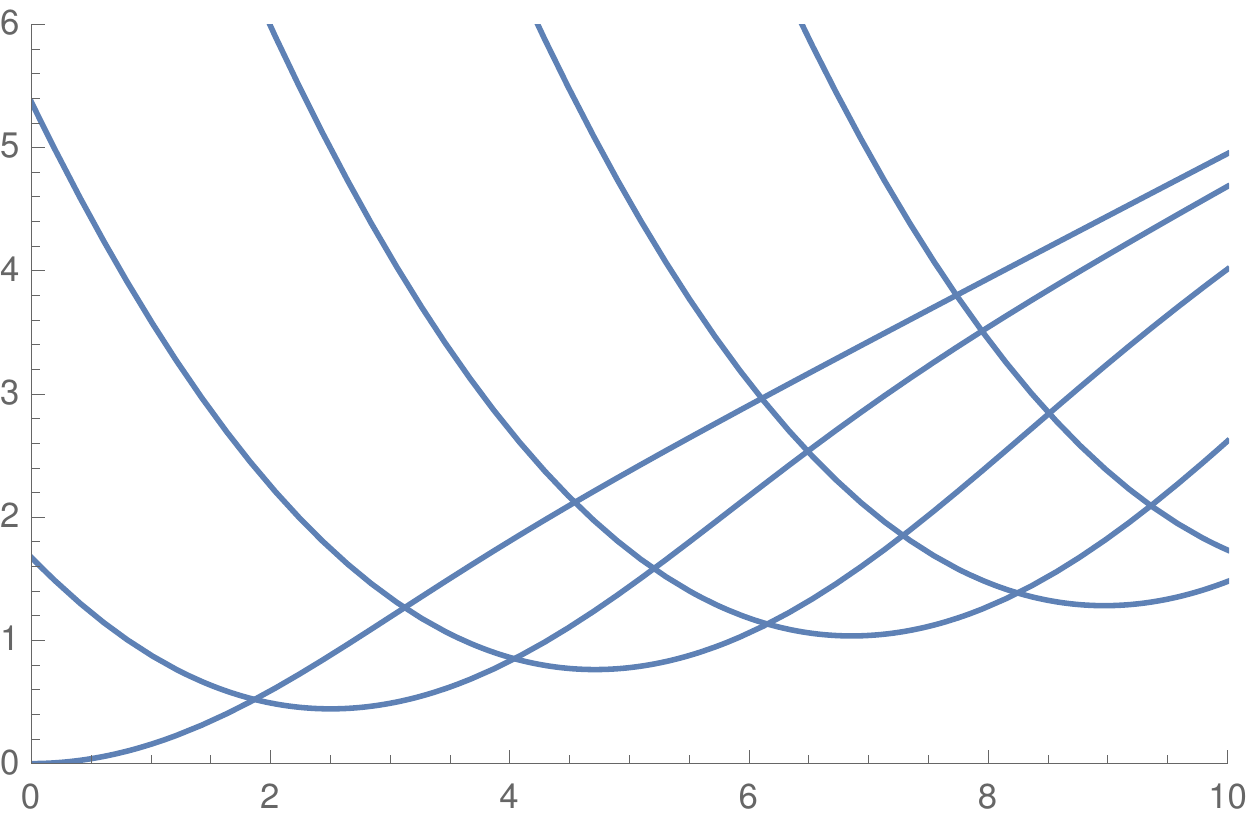}
\caption{The eigenvalues $\lambda_m(b)$ plotted as functions of $b$. The graph indicates a non-diamagnetic effect: the function $b\mapsto\inf_{m\in\mathbb Z}\lambda_m(b)$ is not monotonic. } \label{fig1} 
\end{figure}
\end{center}

{\bf Acknowledgments} The authors would like to thank M.P. Sundqvist for  Fig.~\ref{fig1}. This work started when AK visited the Laboratoire Jean Leray (Nantes) in January 2020 with the financial support of the programme D\'efimaths (supported by the region Pays de la Loire).  The research of AK is partially supported by the Lebanese University within the project
``Analytical and numerical aspects of the Ginzburg-Landau model''.

\bibliographystyle{abbrv}
\bibliography{biblio.bib}
\end{document}